\documentclass{amsart}

\DeclareMathOperator{\ad}{ad}

\DeclareMathOperator{\lie}{\mathfrak{lie}}
\DeclareMathOperator{\der}{\mathfrak{der}}
\DeclareMathOperator{\dert}{\mathfrak{tder}}
\DeclareMathOperator{\Dert}{{\rm TAut}}
\DeclareMathOperator{\Sder}{{\rm SAut}}
\DeclareMathOperator{\sder}{\mathfrak{sder}}
\DeclareMathOperator{\kv}{\mathfrak{kv}}
\DeclareMathOperator{\KV}{{\rm KV}}
\DeclareMathOperator{\hkv}{\widehat{\kv}}
\DeclareMathOperator{\HKV}{\widehat{\KV}}
\DeclareMathOperator{\grt}{\mathfrak{grt}}
\DeclareMathOperator{\dmr}{\mathfrak{dmr}}
\DeclareMathOperator{\Ass}{Ass}
\DeclareMathOperator{\LLie}{{\rm Lie}}
\DeclareMathOperator{\GRT}{{\rm GRT}}
\DeclareMathOperator{\ch}{ch}
\DeclareMathOperator{\tr}{\mathfrak{tr}}
\DeclareMathOperator{\Tr}{{\rm tr}}
\DeclareMathOperator{\Sol}{{\rm Sol}}

\newcommand\Lie[1]{\mathfrak{#1}}
\newcommand{\g}{\Lie{g}}

\renewcommand{\t}{\Lie{t}}

\renewcommand{\k}{\mathbb{K}}

\newcommand{\R}{\mathbb{R}}
\newcommand{\Z}{\mathbb{Z}}

\newcommand{\tidelta}{\tilde{\delta}}
\newcommand{\dv}{{\rm div}}
\renewcommand{\d}{{\rm d}}

\theoremstyle{plain}
\newtheorem{theorem}{Theorem}[section]

\newtheorem{proposition}{Proposition}[section]

\theoremstyle{definition}
\newtheorem{definition}{Definition}[section]

\newtheorem{remark}{\em Remark}[section]

\newtheorem{example}{\em Example}[section]

\begin{document}

\title[]{The Kashiwara-Vergne conjecture and Drinfeld's associators}

\author{Anton Alekseev}
\address{Section de math\'ematiques, Universit\'e de Gen\`eve, 2-4 rue du Li\`evre,
c.p. 64, 1211 Gen\`eve 4, Switzerland}
\email{alekseev@math.unige.ch}

\author{Charles Torossian}
\address{Universit\'e Denis-Diderot-Paris 7, UFR de math\'ematiques, Site Chevaleret, 
Case 7012, 75205 Paris cedex 13, France }
\email{torossian@math.jussieu.fr}


\begin{abstract}
The Kashiwara-Vergne (KV) conjecture is a property of the Campbell-Hausdorff
series put forward in 1978, in \cite{kv}. It has been settled in the positive
by E.~Meinrenken and the first author in 2006,  in \cite{am}. In this paper, we
study the uniqueness issue for the KV problem. To this end, we introduce
a family of infinite dimensional groups $\KV_n$, and an extension $\HKV_2$ 
of the group $\KV_2$. We show that the group $\HKV_2$ contains the 
Grothendieck-Teichm\"uller group $\GRT$ as a subgroup, and that it acts 
freely and transitively on the set of solutions of the KV problem $\Sol(\HKV)$. 
Furthermore,  we prove that $\Sol(\HKV)$ is isomorphic to a direct product
of a line $\k$ ($\k$ being a field of characteristic zero) and the set
of solutions of the pentagon equation with values in the group $\KV_3$.
The latter contains the set of Drinfeld's associators as a subset. 
As a by-product of our construction, we obtain a new proof of the Kashiwara-Vergne
conjecture based on the Drinfeld's theorem on existence of associators.
\end{abstract}


\maketitle

\section{Introduction}

The Kashiwara-Vergne (KV) conjecture is a property of the Campbell-Hausdorff
series which was put forward in \cite{kv}. The KV conjecture has many implications
in Lie theory and harmonic analysis. Let $\g$ be a finite dimensional
Lie algebra over a field of characteristic zero. The KV conjecture implies
the Duflo theorem \cite{duflo} on the isomorphism between the center of the 
universal enveloping algebra $U\g$ and the ring of invariant polynomials
$(S\g)^\g$. Another corollary of the KV conjecture is
a ring isomorphism in cohomology $H(\g, U\g) \cong H(\g, S\g)$ (proved
by Shoikhet \cite{shoikhet} and by Pevzner-Torossian \cite{pt}) for the 
enveloping and symmetric algebras viewed as $\g$-modules with respect
to the adjoint action. For $\k=\R$, another application of the KV conjecture 
is the extension of the Duflo theorem to germs of invariant distributions
on the Lie algebra $\g$ and on the corresponding Lie group $G$ (see 
Proposition 4.1 and Proposition 4.2 in \cite{kv} proved in \cite{ads}
and \cite{ast}).

The KV conjecture was established for solvable Lie algebras by Kashiwara and
Vergne in \cite{kv}, for $\g={\rm sl}(2, \R)$ by Rouvi\`ere in \cite{rouviere},
and for quadratic Lie algebras (that is, Lie algebras equipped with an invariant nondegenerate
symmetric bilinear form, {\em e.g.} the Killing form for $\g$ semisimple) by
Vergne \cite{vergne}. The general case has been settled by Meinrenken and
the first author in \cite{am} based on the previous work of the second
author \cite{torossian} and on the Kontsevich deformation quantization
theory~\cite{kont}.

In this paper, we establish a relation between the KV conjecture and the
theory of Drinfeld's associators developed in \cite{dr}. To this end, we introduce
a family of infinite dimensional groups $\KV_n, n=2,3, \dots$, and 
an extension $\HKV_2$ of the group $\KV_2$. We show that the
set of solutions of the KV conjecture $\Sol(\HKV)$ carries a free and 
transitive action of the group  $\HKV_2$ which contains the Drinfeld's 
Grothendieck-Teichm\"uller group $\GRT$ as a subgroup. Furthermore,
the set $\Sol(\HKV)$ is isomorphic to a direct product of a line $\k$
and the set of solutions of the pentagon equation with values in the group
$\KV_3$. We make use of an involution $\tau$ acting on solutions of the KV
conjecture to select symmetric solutions of the KV problem, $\Sol^\tau(\HKV)$.
The set $\Sol^\tau(\HKV)$ is isomorphic to a direct product of a line
and the set of associators (joint solutions of the pentagon, hexagon and
inversion equations of \cite{dr}) with values in the group $\KV_3$. The latter
contains the set of Drinfeld's associators as a subset.

In summary, we solve the uniqueness issue for the KV problem
in terms of associators with values in the group $\KV_3$. As a by-product,
we obtain a new proof of the KV conjecture. Indeed, by Drinfeld's theorem, 
the set of Drinfeld's associators in non empty. Hence, so is the set of associators
with values in the group $\KV_3$, and the set of symmetric solutions
of the KV conjecture $\Sol^\tau(\HKV)$. This new proof is based
on the theory of associators rather than on the deformation quantization
machine.

An outstanding question which we were not able to resolve is whether or not
the symmetry group of the KV problem, $\HKV_2$, is isomorphic to a direct 
product of a line and the Grothendieck-Teichm\"uller group $\GRT$. A numerical
experiment of L.~Albert and the second author shows that the corresponding
graded Lie algebras coincide up to degree 16! If correct, the isomorphism
$\HKV_2 \cong \k \times \GRT$ would imply that all solutions of the KV conjecture
are symmetric, and that all associators with values in the group $\KV_3$ are
Drinfeld's associators.

Below we explain {\em  raison d'\^etre} of the link between 
the Kashiwara-Vergne and associator theories. One possible formulation
of the KV problem is as follows: find an automorphism $F$ of the 
(degree completion of the) free Lie
algebra with generators $x$ and $y$ such that
\begin{equation} \label{intro:kv}
F: x+y \mapsto \ch(x,y),
\end{equation}
where $\ch(x,y)=x+y+\frac{1}{2}[x,y] + \dots$ is the Campbell-Hausdorff series.
The automorphism $F$ should satisfy several other properties which we omit here.
Consider a free Lie algebra with three generators $x,y,z$ and define
the automorphism $F^{1,2}$ which is equal to $F$ when acting on generators 
$x$ and $y$ and which preserves the generator $z$. Similarly,
define $F^{2,3}$ acting on generators $y$ and $z$ and preserving $x$.
Furthermore, define $F^{12,3}$ acting on $x+y$ and $z$, and 
$F^{1,23}$ acting on $x$ and $y+z$ (for a precise definition
see Section \ref{sec:derivations}). The main property of the Campbell-Hausdorff series
is the associativity,
$$
\ch(x, \ch(y,z))=\ch(\ch(x,y), z).
$$
We use this property to establish the following formula:
$$
\begin{array}{lll}
F^{1,2}F^{12,3}(x+y+z) & = & F^{1,2}(\ch(x+y,z)) \\
                       & = & \ch(\ch(x,y),z)   \\
                       & = &  \ch(x,\ch(y,z))  \\
                       & = & F^{2,3}(\ch(x, y+z)) \\
                       & = & F^{2,3}F^{1,23}(x+y+z).
\end{array}
$$
Hence, the combination
\begin{equation}  \label{intro:assoc}  
\Phi= (F^{12,3})^{-1} (F^{1,2})^{-1} F^{2,3}F^{1,23}
\end{equation}
has the property $\Phi(x+y+z)=x+y+z$ which is one of the defining properties
of the group $\KV_3$. Furthermore, as an easy consequence of \eqref{intro:kv}
and \eqref{intro:assoc}, the automorphism $\Phi$ satisfies the pentagon
equation
\begin{equation}  \label{intro:pent}
\Phi^{1,2,3} \Phi^{1,23,4} \Phi^{2,3,4} = \Phi^{12,3,4} \Phi^{1,2,34} .
\end{equation}
Equation \eqref{intro:pent} is an algebraic presentation of two
sequences of parenthesis redistributions in a product of four objects
(a standard example is a tensor product in tensor categories):
the left hand side corresponds to a passage 
$((12)3)4 \to (1(23))4 \to 1((23)4) \to (1(2(34))$, while the
right hand side to $((12)3)4 \to ((12)(34)) \to 1(2(34))$.
The pentagon equation is the most important element of
the Drinfeld's theory of associators. Our main technical result
shows that solutions of equation~\eqref{intro:pent} with values
in the group $\KV_3$ admit an almost unique decomposition
of the form \eqref{intro:assoc}, and the corresponding 
automorphism $F$ is automatically a solution of the KV problem 
(and, in particular, has the property
\eqref{intro:kv}).

An important object of the Kashiwara-Vergne theory is the Duflo
function $J^{1/2}$ which corrects the symmetrization map 
${\rm sym}: S\g \to U\g$ so as it restricts to a ring isomorphism
on $\ad_\g$-invariants. It is more convenient to discuss the
logarithm of the Duflo function,
\begin{equation}  \label{intro:duflo}
f(x)=\frac{1}{2} \ln\left( \frac{e^{x/2}-e^{-x/2}}{x} \right) =
\frac{1}{2} \sum_{k=2}^\infty \, \frac{B_k}{k \cdot k!} \, x^k, 
\end{equation}
where $B_k$ are Bernoulli numbers. The function $f(x)$ is even,
and it is known that any function $\tilde{f}(x)=f(x)+h(x)$ with $h(x)$ odd still
defines a ring isomorphism between $Z(U\g)$ and $(S\g)^\g$
(in the category of Lie algebras, all these isomorphisms
coincide with the Duflo isomorphism). We show that the Drinfeld's
generators $\sigma_{2k+1}, k=1,2,\dots$ of the Grothendieck-Teichm\"uller
Lie algebra $\grt$ define flows on the set of solutions of the KV 
conjecture $\Sol(\HKV)$, and on the odd parts of Duflo functions 
such that $(\sigma_{2k+1} \cdot h)(x)=-x^{2k+1}$. Hence, all odd
formal power series (the linear term of the Duflo function is not
well defined) $h(x)$ can be reached by the action of the group $\GRT$
on the symmetric Duflo function \eqref{intro:duflo}.
This action coincides with the one described in \cite{kont'} (see Theorem 7).

The plan of the paper is as follows: in Section \ref{sec:free} we 
introduce a Hochschild-type cohomology theory for free Lie algebras, 
compute the cohomology in low degrees (Theorem \ref{th:hdelta}), and discuss
the associativity property of the Campbell-Hausdorff series. In Section
\ref{sec:derivations} we study derivations of free Lie algebras.
Again, we define a Hochschild-type cohomology theory, and compute
cohomology in low degrees (Theorem \ref{th:hd}). In Section \ref{sec:KVlie}
we introduce a family of Kashiwara-Vergne Lie algebras $\kv_n$ and the
Lie algebra $\hkv_2$, and show that the Grothendieck-Teichm\"uller Lie
algebra $\grt$ injects into $\hkv_2$ (Theorem \ref{th:grt}).
In Section \ref{sec:kv} we give a new formulation of the Kashiwara-Vergne
conjecture, and show that it is equivalent to the original
statement of \cite{kv} (Theorem \ref{th:old=new}).
In Section \ref{sec:duflo} we discuss properties of Duflo
functions and show that they can acquire arbitrary odd parts.
In Section \ref{sec:pent} we establish a link between solutions
of the KV problem and solutions of the pentagon equation with values
in the group $\KV_3$ (Theorem \ref{th:pent=>kv}).
In Section \ref{sec:sym} we discuss an involution $\tau$ on the
set of solutions of the KV problem, and derive the hexagon equations 
using this involution. Finally, in Section \ref{sec:associators}
we introduce associators with values in the group $KV_3$, compare them
to Drinfeld's associators, and give a new proof of the KV conjecture 
(Theorem \ref{th:kvneq0}).

\vskip 0.3cm

{\bf Acknowledgements:} We are indebted to V.~Drinfeld for posing us a question
on the relation between solutions of the Kashiwara-Vergne conjecture and associators.
We are grateful to L.~Albert who helped setting up numerical experiments
which influenced this project in a significant way.
We thank D.~Bar-Natan, M.~Duflo, A.~Joseph, I.~Marin, E.~Meinrenken and S.~Sahi
for useful discussions and remarks. 
Research of A.A. was supported in part by the Swiss National Science Foundation.
Research of C.T. was supported by CNRS.

\section{Free Lie algebras} \label{sec:free}

\subsection{Lie algebras $\lie_n$ and the Campbell-Hausdorff series.} \label{subsec:lie}
Let $\k$ be a field of characteristic zero, and let $\lie_n=\lie(x_1,\dots, x_n)$ 
be the degree completion of the free Lie algebra over $\k$ with generators
$x_1, \dots, x_n$. It is a graded Lie algebra
$$
\lie_n = \prod_{k=1}^\infty \lie^k(x_1, \dots, x_n),
$$
where $\lie^k(x_1, \dots, x_n)$ is spanned by Lie words consisting
of $k$ letters.  In case of $n=1,2,3$ we shall often denote 
the generators by $x,y,z$. 

The universal enveloping algebra of $\lie_n$ is
the degree completion of the free associative
algebra with generators $x_1, \dots, x_n$, $U(\lie_n)=\Ass_n$.  
Every element $a \in \Ass_n$ has a unique decomposition
\begin{equation} \label{eq:partial}
a=a_0 + \sum_{k=1}^n (\partial_k a) x_k,
\end{equation}
where $a_0\in \k$ and $(\partial_k a) \in \Ass_n$. 

The Campbell-Hausdorff series is an element of $\Ass_2$ defined by
formula $\ch(x,y)=\ln(e^xe^y)$, where $e^x=\sum_{k=0}^\infty x^k/k!$ and
$\ln(1-a)=-\sum_{k=1}^\infty a^k/k$. By Dynkin's theorem \cite{dynkin},
$\ch(x,y) \in \lie_2$ and
$$
\ch(x,y)=x+y+\frac{1}{2}[x,y] + \dots ,
$$
where $\dots$ stands for a series in multiple Lie brackets in $x$ and $y$.
The Campbell-Hausdorff series satisfies the associativity property
in $\lie_3$,
\begin{equation} \label{eq:assocch}
\ch(x,\ch(y,z))=\ch(\ch(x,y),z). 
\end{equation}
One can rescale the Lie bracket of $\lie_2$ by posing
$[\cdot, \cdot]_s=s[\cdot , \cdot]$ for $s\in \k$ to obtain a rescaled Campbell-Hausdorff
series,
$$
\ch_s(x,y)=x+y+\frac{s}{2} [x,y] + \dots ,
$$
where elements of $\lie^k(x,y)$ get a extra factor of $s^{k-1}$. Note
that $\ch_s(x,y)=s^{-1}\ch(sx,sy)$ and $\ch_0(x,y)=x+y$. The rescaled
Campbell-Hausdorff series $\ch_s(x,y)$ satisfies the associativity
equation,
$$
\begin{array}{lll}
\ch_s(x,\ch_s(y,z)) & = & s^{-1} \ch(sx, \ch(sy,sz)) \\
 & = & s^{-1} \ch(\ch(sx,sy),sz) \\
& = & \ch_s(\ch_s(x,y),z).
\end{array}
$$

\begin{remark}
Let $\g$ be a finite dimensional Lie algebra over $\k$. Then, every element
$a \in \lie_n$ defines a formal power series $a_\g$ on $\g^n$ with values
in $\g$. For instance, the Campbell-Hausdorff series $\ch \in \lie_2$
defines a formal power series $\ch_\g$ on $\g^2$  with rational coefficients.
For every finite dimensional Lie algebra $\g$ this formal power series has
a finite convergence radius. 
\end{remark}

\subsection{The vector space $\tr_n$.}  \label{subsec:tr}
For every $n$ we define a graded vector space $\tr_n$ as a quotient
$$
\tr_n = \Ass^+_n/\langle (ab-ba); a,b \in \Ass_n \rangle .
$$
Here $\Ass^+_n=\prod_{k=1}^\infty \Ass^k(x_1, \dots, x_n)$, and
$\langle (ab-ba); a,b \in \Ass_n \rangle$ is the subspace of $\Ass^+_n$ spanned
by commutators. Product of $\Ass_n$ does not
descend to $\tr_n$ which only has a structure of a graded vector
space. We shall denote by $\Tr: \Ass_n \to \tr_n$ the natural
projection. By definition, we have $\Tr(ab)=\Tr(ba)$ for all $a,b \in \Ass_n$
imitating the defining property of trace.

\begin{example}
The space $\tr_1$ is isomorphic to the space of formal power series
in one variable without constant term, $\tr_1 \cong x \k[[x]]$.
This isomorphism is given by the following formula,
$$
f(x)=\sum_{k=1}^\infty f_k x^k \mapsto  \sum_{k=1}^\infty f_k \Tr(x^k) .
$$
\end{example}

In general, graded components $\tr^k_n$ of the space $\tr_n$ are spanned by
words of length $k$ modulo cyclic permutations.

\begin{example}
For $n=2$, $\tr_2^1$ is spanned by $\Tr(x)$ and $\Tr(y)$,
$\tr_2^2$ is spanned by $\Tr(x^2), \Tr(y^2)$ and $\Tr(xy)=\Tr(yx)$,
$\tr_2^3$ is spanned by $\Tr(x^3)$, $\Tr(x^2y)$, $\Tr(xy^2)$ and  $\Tr(y^3)$,
$\tr_2^4$ is spanned by $\Tr(x^4)$, $\Tr(x^3y)$, $\Tr(x^2y^2)$, $\Tr(xyxy)$, $\Tr(xy^3)$ and
$\Tr(y^4)$ {\em etc.}
\end{example}

\begin{remark}
Let $\g$ be a finite dimensional Lie algebra over $\k$, $\rho: \g \to {\rm End}(V)$
be a finite dimensional representation of $\g$, and $a=\sum_{k=1}^\infty a_k \in \tr_n$
an element of $\tr_n$.
We define $\rho(a)$ as a formal power series on $\g^n$ such that
$\rho(\tr(x_{i_1} \dots x_{i_k}))={\rm Tr}_V(\rho(x_{i_1}) \dots \rho(x_{i_k}))$
for monomials, and this definition extends by linearity to all elements of $\tr_n$. 
\end{remark}

\subsection{Cohomology problems in $\lie_n$ and $\tr_n$.} \label{subsec:coh}
For all $n=1,2, \dots$ we define an operator $\delta: \lie_n \to \lie_{n+1}$ by formula
\begin{equation} \label{eq:delta}
\begin{array}{lll}
(\delta f)(x_1, \dots, x_{n+1}) & = & f(x_2,x_3, \dots, x_{n+1}) \\
& + & \sum_{i=1}^{n} (-1)^i f(x_1, \dots, x_i+x_{i+1}, \dots, x_{n+1}) \\
& + & (-1)^{n +1}f(x_1, \dots, x_{n}).
\end{array} 
\end{equation}
It is easy to see that $\delta^2=0$.

\begin{example} \label{ex:deltalie}
For $n=1$ and $f=ax \in \lie_1 \cong \k$ we have
$$
(\delta f)(x,y)=f(x)-f(x+y)+f(y)=0.
$$
For $n=2$ we get
$$
(\delta f)(x,y,z)=f(y,z)-f(x+y,z)+f(x,y+z)-f(x,y).
$$
\end{example}

One can also use equation \eqref{eq:delta} to define a differential on 
the family for vector spaces $\tr_n$. By abuse of notations, we denote it by
the same letter, $\delta: \tr_n \to \tr_{n+1}$.

\begin{example} \label{ex:deltatr}
For $n=1$, we have for $f(x)=\Tr(x^k)$
$$
(\delta f)(x,y)=\Tr(x^k+y^k-(x+y)^k).
$$
Note that the right hand side vanishes for $k=1$ and that it is 
non-vanishing for all other $k=2,3 \dots$.
\end{example}

The following theorem gives the cohomology of $\delta$ in degrees $n=1,2$.

\begin{theorem} \label{th:hdelta}
$$
\begin{array}{lll}
H^1(\lie, \delta) & = & {\rm ker}\, (\delta: \lie_1 \to \lie_2) = \lie_1 \, ,\\
H^1(\tr, \delta)  & = & {\rm ker}\, (\delta: \tr_1 \to \tr_2)\cong \k \, \Tr(x) \, , \\
H^2(\lie, \delta) & \cong &  [\k [x,y] \, ] \, , \\
H^2(\tr, \delta) & = & 0 \, .
\end{array}
$$
\end{theorem}

\begin{proof}
The first statement is obvious since $\lie_1= \k x$ and 
$\delta(x)=x-(x+y)+y=0$. The second statement follows from the
calculation of Example \ref{ex:deltatr}. 

For computing the second cohomology, let $f$ be a solution of
degree $n\geq 2$ of equation 
\begin{equation} \label{eq:4term}
f(y,z)-f(x+y,z)+f(x,y+z)-f(x,y)=0.
\end{equation}
By putting
$x \mapsto sx, y \mapsto x, z \mapsto z$ we obtain
$$
f(sx,x)+f((1+s)x,z)-f(sx,x+z)-f(x,z)=0.
$$
In a similar fashion, putting $x \mapsto x, y\mapsto z, z \mapsto sz$ yields
$$
f(x,z)+f(x+z,sz)-f(x,(1+s)z)-f(sz,z)=0.
$$
Subtracting the first equation from the second one
and differentiating the result in $s$  gives
\begin{equation} \label{eq:nf}
\begin{array}{lll}
n f(x,z) & = & \frac{d}{ds} \, (f((1+s)x,z)+f(x,(1+s)z))|_{s=0} \\
& = &\frac{d}{ds} \, (f(sx,x+z)+f(x+z,sz)-f(sx,x)-f(sz,z))|_{s=0}.
\end{array}
\end{equation}
First, we solve equation \eqref{eq:nf} for $f \in \lie_2$. In this case, 
$f(sx,x)=f(sz,z)=0$ and we obtain
$$
f(x,z)=\ad_{x+z}^{n-1}(\alpha x+ \beta z)
$$
for some $\alpha, \beta \in \k$.
For $n=2$, this yields $f(x,z)=(\beta-\alpha)[x,z]$.
It is easy to check that this is a solution of equation \eqref{eq:4term}.

For $n\geq 3$, consider equation \eqref{eq:4term} and 
first put $y=-z$ to get $f(x,z)=-f(x-z,z)$,
and then put $y=-x$ to obtain $f(x,z)=-f(-x,x+z)$. Hence,
$$
f(x,z)=(\alpha-\beta)\ad_x^{n-1}z =(\alpha-\beta) \ad_z^{n-1}x
$$
which implies $f(x,z)=0$. Finally, for $n=1$ we put $f(x,y)=\alpha x+ \beta y$ to 
obtain $\delta f=\alpha x- \beta z$. In conclusion, $\delta f =0$ implies
that $f$ is of degree two, and $f(x,y)= \alpha [x,y]$ for $\alpha \in \k$.

For $f \in \tr_2$ equation \eqref{eq:nf} gives
$$
f(x,z)=\Tr \, \left( (\alpha x+ \beta z)(x+z)^{n-1} - \alpha x^n - \beta z^n \right) \, ,
$$
for some $\alpha,\beta \in \k$.
For $n=1$, it implies $f(x,z)=0$. For $n=2$, we get
$$
f(x,z)=(\alpha+\beta) \Tr (xz) = - \, \frac{\alpha+\beta}{2} \delta(\Tr(x^2)) \, .
$$
For $n\geq 3$, we have
$$
\delta f=(\beta-\alpha)\Tr \, y((x+y)^{n-1}+(y+z)^{n-1}-(x+y+z)^{n-1}-y^{n-1}) \, .
$$
The coefficient in front of $\Tr(y^{n-2}xz)$ in this expression is equal to
$(\beta-\alpha)(n-2)$, and it vanishes if and only if $\beta=\alpha$. In this case,
$f(x,z)=-\alpha \delta(\Tr(x^n))$. Hence, $\delta f=0$ implies the existence
of $g\in \tr_1$ such that $\delta g =f$, and the second cohomology 
$H^2(\tr, \delta)$ vanishes.
\end{proof}

\begin{remark} 
In the proof of Theorem \ref{th:hdelta} we have shown that
${\rm ker}\, (\delta: \lie_2 \to \lie_3) = \k [x,y]$. 
That is, the only solution of equation \eqref{eq:4term} is $f(x,y)=\alpha [x,y]$.
Equation \eqref{eq:4term} has been previously considered in 
the proof of Proposition 5.7 in \cite{dr}. There it is stated that equation
\eqref{eq:4term} has no nontrivial symmetric, $f(x,y)=f(y,x)$, 
solutions in $\lie_2$. 
\end{remark}


\subsection{Applications}   \label{subsec:app}
In this section we collect two simple applications of the 
cohomology computations of Section \ref{subsec:coh}.

\begin{proposition} \label{prop:chunique}
Let $s \in \k$ and let $\chi \in \lie_2$ be a Lie series of the form
$\chi(x,y)=x+y+\frac{s}{2} \, [x,y] + \dots$, where $\dots$ stand for a series
in multibrackets. Assume that $\chi$ is associative, that is
$$
\chi(x, \chi(y,z))=\chi(\chi(x,y), z) \, \in \, \lie_3 \, .
$$
Then, $\chi$ coincides with the rescaled Campbell-Hausdorff series,
$\chi(x,y)=\ch_s(x,y)$.
\end{proposition}

\begin{proof}
The Lie series $\chi$ and $\ch_s$ coincide up to degree 2.
Assume that they coincide up to degree $n-1$, and let
$\chi=\sum_{n=1}^\infty \chi_n$ with $\chi_n(x,y)$ a Lie polynomial
of degree $n$. The associativity equation implies the following equation
for $\chi_n$:
$$
\chi_n(x, y+z)+\chi_n(y,z)-\chi_n(x,y)-\chi_n(x+y,z)=
\mathcal{F}(\chi_1(x,y), \dots, \chi_{n-1}(x,y)),
$$
where $\mathcal{F}$ is a certain (nonlinear) function of the lower
degree terms. By the induction hypothesis, the lower degree terms
of $\chi$ and $\ch_s$ coincide. And the equation for $\chi_n$ has a
unique solution since the only solution of the corresponding homogeneous 
equation $\delta \chi_n=0$ for $n \geq 3$ is $\chi_n=0$. 
Hence, $\chi_n=(\ch_s)_n$ and $\chi=\ch_s$. 
\end{proof}

Similar to the differential $\delta$, we introduce another 
differential $\tidelta$ acting on $\lie_n$ and $\tr_n$:
\begin{equation} \label{eq:tidelta}
\begin{array}{lll}
(\tidelta f)(x_1, \dots, x_{n+1}) & = & f(x_2,x_3, \dots, x_{n+1}) \\
& + & \sum_{i=1}^{n} (-1)^i f(x_1, \dots, \ch(x_i,x_{i+1}), \dots, x_{n+1}) \\
& + & (-1)^{n +1}f(x_1, \dots, x_{n}).
\end{array} 
\end{equation}
Again, $\tidelta^2=0$, but
in contrast to $\delta$, $\tidelta$ does not preserve the degree.
In the following proposition we compute the cohomology of $\tidelta$
for $n=1,2$.

\begin{proposition} \label{prop:htidelta}
$$
\begin{array}{lll}
H^1(\lie, \tidelta) & = & 0  \, ,\\
H^1(\tr, \delta)  & = & {\rm ker}\, (\tidelta: \tr_1 \to \tr_2)\cong
\k \, \Tr(x) \, , \\
H^2(\lie, \delta) & = & 0 \, , \\
H^2(\tr, \delta) & = & 0 \, .
\end{array}
$$
\end{proposition}

\begin{proof}
For $H^1(\lie, \tidelta)$ we consider
$\tidelta(x)=x+y-\ch(x,y) \neq 0$ which implies
$H^1(\lie, \tidelta)={\rm ker}(\tidelta: \lie_1 \to \lie_2)=0$.
To compute $H^1(\tr, \delta)$, observe that 
$\tidelta(\Tr(x))=\Tr(x+y-\ch(x,y))=0$ (here we used that
$\Tr(a)=0$ for all $a \in \lie_n$ of degree greater or equal to two),
and $\tidelta \Tr(x^k)= \delta \Tr(x^k) + \dots \neq 0$ for $k\geq 2$
(here $\dots$ stand for the terms of degree greater than $k$).

In order to compute the second cohomology, let $f=\sum_{n=k}^\infty f_n$, 
where $f_n$ is homogeneous of degree $n$, and $f_k \neq 0$.
Then, $\tidelta f=\delta f_k + {\rm terms} \, \, {\rm of} \,\, 
{\rm degree} > k$, and $\tidelta f=0$ implies $\delta f_k=0$.

First, consider $f\in \lie_2$. In this case, $\delta f_k=0$ implies
$f_k=0$ for all $k$ except $k=2$. For $k=2$, we have $f_2(x,y)= \frac{\alpha}{2} [x,y]$ for
some $\alpha \in \k$. Define $g=f+\alpha (\tidelta x)=f+\alpha(x+y-\ch(x,y))$. We have
$\tidelta g=\tidelta f+\alpha \tidelta^2 x=0$, and $g_2(x,y)=0$. Hence,
$g=0$ and $f=-\alpha(x+y-\ch(x,y))=\tidelta(-\alpha x)$.

For $f\in \tr_2$, equation $\delta f_k=0$ implies $f_k=\delta h_k$ for some
$h_k \in \tr_1$. Consider $g=f-\tidelta h_k$. It satisfies $\tidelta g=0$,
and $g=\sum_{n=k+1}^\infty g_k$. In this way, we inductively construct $h \in \tr_1$
such that $g=\tidelta h$.
\end{proof}

\begin{remark}
For every $s\in \k$ one can introduce a differential $\tidelta_s$ by replacing
$\ch(x,y)$ with $\ch_s(x,y)$ in formula \eqref{eq:tidelta}. We have $\tidelta_1=\tidelta$
and $\tidelta_0=\delta$. Proposition \ref{prop:htidelta} applies to all $s \neq 0$.
Note that $H^1(\tr, \tidelta_s)=\k \Tr(x)$ and $H^2(\tr, \tidelta_s)=0$ for
all $s \in \k$ (including $s=0$).
\end{remark}

\section{Derivations of free Lie algebras}  \label{sec:derivations}

\subsection{Tangential and special derivations} \label{subsec:tangetial}
We shall denote by $\der_n$ the Lie algebra of derivations 
of $\lie_n$. An element $u \in \der_n$ is completely determined by its 
values on the generators, $u(x_1), \dots, u(x_n) \in \lie_n$. 
The Lie algebra $\der_n$ carries a grading induced by the one of $\lie_n$.

\begin{definition}
A derivation $u \in \der_n$ is called {\em tangential} if there exist
$a_i \in \lie_n, i=1, \dots, n$ such that $u(x_i)=[x_i, a_i]$. 
\end{definition}

Another way to define  tangential derivations is
as follows: for each $i=1,\dots,n$ there exists an inner derivation
$u_i$ such that $(u-u_i)(x_i)=0$.
We denote the subspace of tangential derivations by $\dert_n \subset \der_n$.

\begin{remark}  \label{rem:normalize}
Let $p_k: \lie_n \to \k$ be a projection which assigns to an element
$a=\sum_{k=1}^n \lambda_k x_k + \dots$, where $\dots$ stand for 
multibrackets, the coefficient $\lambda_k \in \k$.
Elements of $\dert_n$ are in one-to-one correspondence with $n$-tuples
of elements of $\lie_n$, $(a_1, \dots, a_n)$, which satisfy the condition
$p_k(a_k)=0$ for all $k$. Indeed, 
the kernel of the operator $\ad_{x_k}: a \mapsto [x_k, a]$ is exactly
$\k x_k$. Hence, an $n$-tuple $(a_1, \dots, a_n)$ defines a vanishing
derivation $u(x_k)=[x_k, a_k]=0$ if and only if $a_k \in \k x_k$ for all $k$.
By abuse of notations, we shall often write $u=(a_1, \dots, a_n)$.
\end{remark}

\begin{proposition}
Tangential derivations form a Lie subalgebra of $\der_n$.
\end{proposition}

\begin{proof}
Let $u=(a_1, \dots, a_n)$ and $v=(b_1, \dots, b_n)$. 
We have
$$
\begin{array}{lll}
[u,v](x_k) &= & u([x_k, b_k])-v([x_k, a_k])  \\
 & = & [[x_k, a_k], b_k] + [x_k, u(b_k)] - [[x_k, b_k], a_k] - [x_k, v(a_k)] \\
 & = & [x_k, u(b_k)-v(a_k) + [a_k, b_k]]
\end{array}
$$
which shows $[u,v] \in \dert_n$.
\end{proof}

One can transport the Lie bracket of $\dert_n$ to the set of $n$-tuples
$(a_1, \dots, a_n)$ which satisfy the condition $p_k(a_k)=0$. Indeed, 
put the $k$th component of the new $n$-tuple equal to 
$u(b_k)-v(a_k) + [a_k, b_k]$. This expression does not contain linear
terms, and in particular it is in the kernel of $p_k$.

\begin{definition}
A derivation $u \in \dert_n$ is called {\em special} if 
$u(x)=0$ for $x=\sum_{i=1}^n x_i$. 
\end{definition}

We shall denote the space of special derivations of $\lie_n$ by $\sder_n$.
It is obvious that $\sder_n \subset \dert_n$ is a Lie subalgebra. 
Indeed, for  $u,v \in \sder_n$ we have $[u,v](x)=u(v(x))-v(u(x))=0$ 
and, hence, $[u,v] \in \sder_n$.

\begin{remark}
Ihara \cite{Ihara} calls elements of $\sder_n$ normalized
special derivations.
\end{remark}

\begin{example}
Consider $r=(y,0) \in \dert_2$. By definition, $r(x)=[x,y], r(y)=0$.
Note that $r(x+y)=[x,y] \neq 0$ and $r \notin \sder_2$. Consider another
element $t=(y,x) \in \dert_2$. We have $t(x)=[x,y], t(y)=[y,x]$ and
$t(x+y)=[x,y]+[y,x]=0$. Hence, $t\in \sder_2$.
\end{example}

\subsection{Simplicial and coproduct maps} \label{subsec:simplicial}
We shall need a number of Lie algebra homomorphisms mapping
$\dert_{n-1}$ to $\dert_n$. First, observe that the permutation group
$S_n$ acts on $\lie_n$ by Lie algebra automorphisms. 
For $\sigma \in S_n$, we have
$a \mapsto a^\sigma=a(x_{\sigma(1)}, \dots, x_{\sigma(n)})$.
The induced action on $\dert_n$ is given by formula,
$$
u=(a_1, \dots, a_n) \mapsto u^\sigma = 
(a_{\sigma^{-1}(1)}(x_{\sigma(1)}, \dots, x_{\sigma(n)}), \dots, 
a_{\sigma^{-1}(n)}(x_{\sigma(1)}, \dots, x_{\sigma(n)})).
$$

\begin{example}
For $u=(a(x,y), b(x,y)) \in \dert_2$ we have 
$u^{2,1}=(b(y,x),a(y,x))$, where $\sigma=(21)$ is the nontrivial
element of $S_2$. In the same fashion, for $u=(a(x,y,z),b(x,y,z),c(x,y,z)) 
\in \dert_3$ we have $u^{3,1,2}=(b(z,x,y),c(z,x,y),a(z,x,y))$.
\end{example}

We define {\em simplicial maps} by the following
property. For $u=(a_1, \dots, a_{n-1}) \in \dert_{n-1}$ define
$u^{1,2,\dots,n-1}=(a_1,\dots,a_{n-1},0)\in \dert_n$. It is clear that the
map $u \mapsto u^{1,2,\dots,n-1}$ is a Lie algebra homomorphism. 
We obtain other simplicial maps by composing with the action of 
$S_n$ on $\dert_n$. Simplicial maps restrict to special derivations.
Indeed, for $u \in \sder_{n-1}$ and $x=\sum_{i=1}^n x_i$ we compute
$$
u^{1,2,\dots,n-1}(x) = \sum_{i=1}^{n-1} [x_i, a_i]=0
$$
which implies $u^{1,2,\dots,n-1} \in \sder_n$.

\begin{example}
For $u=(a(x,y),b(x,y)) \in \dert_2$ we have
$u^{1,2}=(a(x,y),b(x,y),0) \in \dert_3$ and 
$u^{2,3}=(0,a(y,z),b(y,z))$. For instance, for
$r=(y,0)$ we obtain $r^{1,2}=(y,0,0), r^{2,3}=(0,z,0), r^{1,3}=(z,0,0)$.
\end{example}

\begin{proposition}   \label{prop:cybe}
The element $r=(y,0) \in \dert_2$ satisfies the classical
Yang-Baxter equation,
$$
[r^{1,2},r^{1,3}]+[r^{1,2},r^{2,3}]+[r^{1,3},r^{2,3}]=0.
$$
\end{proposition}

\begin{proof}
We compute,
$$
[ r^{1,2} , r^{1,3} ] =  [(y,0,0),(z,0,0)] = ([y,z],0,0),
$$
$$
[ r^{1,2} , r^{2,3} ]  = [(y,0,0),(0,z,0)] = -([y,z],0,0),
$$
$$
[ r^{1,3} , r^{2,3} ] =  [(z,0,0),(0,z,0)] = 0.    
$$
Adding these expressions gives zero, as required.
\end{proof}

Next, consider $t=(y,x) \in \sder_2$. By 
composing various simplicial maps we obtain $n(n-1)/2$ elements
of $t^{i,j}=t^{j,i} \in \dert_n$ with non-vanishing components $x_i$ at the $j$th
place and $x_j$ at the $i$th place.

\begin{proposition}  \label{prop:tn}
Elements $t^{i,j} \in \sder_n$ span a Lie subalgebra isomorphic
to the quotient of the free Lie algebra with $n(n-1)/2$ generators
by the following relations,
\begin{equation} \label{eq:t1}
[ t^{i,j}, t^{k,l}] =0
\end{equation}
for $k,l \neq i,j$, and 
\begin{equation} \label{eq:t2}
[t^{i,j}+t^{i,k}, t^{j,k}]=0
\end{equation}
for all triples of distinct indices $i,j,k$.
\end{proposition}

\begin{remark}
We denote by $\t_n$ the Lie algebra defined by relations \eqref{eq:t1} and~\eqref{eq:t2}. 
Note that $c=\sum_{i<j} t^{i,j}$ is a central element of $\t_n$. Indeed,
$[t^{i,j},c]=\sum_{k\neq i, k\neq j}[t^{i,j}, t^{i,k}+t^{j,k}]=0$.
It is known (see Section 5 of \cite{dr}) that 
$$
\t_n \cong \t_{n-1} \oplus \lie(t^{1,n}, \dots, t^{n-1,n}),
$$
where the free Lie algebra $\lie(t^{1,n}, \dots, t^{n-1,n})$ is an ideal 
in $\t_n$ and $\t_{n-1} \subset \t_n$ is a complementary Lie subalgebra spanned
by $t^{i,j}$ with $i,j<n$ . In particular, $\t_2 = \k t^{1,2}$ is an abelian 
Lie algebra with one generator, and 
$\t_3\cong \t_2 \oplus \lie(t^{1,3}, t^{2,3})$. In fact, ${\rm ad}_{t^{1,2}}$ is
an inner derivation of $\lie(t^{1,3}, t^{2,3})$,
$$
[t^{1,2}, a]=[t^{1,2}-c,a]=-[t^{1,3}+t^{2,3},a],
$$
and  $\t_3 \cong \k c \oplus \lie(t^{1,3}, t^{2,3})$. 
\end{remark}

\begin{proof}
First, we verify the relations \eqref{eq:t1} and \eqref{eq:t2}. The first one
is obvious since the derivations $t^{i,j}$ and $t^{k,l}$ act on different
generators of $\lie_n$. For the second one, we choose $n=3$ and compute
$[t^{1,2}+t^{1,3}, t^{2,3}]$:
$$
[t^{1,2}, t^{2,3}]  =  [(y,x,0),(0,z,y)]  =  (-[y,z],[x,z],[y,x]), 
$$
$$
[t^{1,3}, t^{2,3}] =  [(z,0,x),(0,z,y)]  =  (-[z,y],[z,x],[x,y]).
$$
Adding these expressions gives zero, as required. We obtain the relation
\eqref{eq:t2} for other values of $i,j,k$ by applying the $S_n$ action to
replace $1,2,3$ by $i,j,k$. Hence, the expressions
$t^{i,j}$ define a Lie algebra homomorphism from $\t_n$ to $\sder_n$.
We prove that it is injective by induction. Clearly, the map 
$\t_2=\k t^{1,2} \to \sder_2$ is injective. Assume that the
Lie homomorphism $\t_{n-1} \to \dert_{n-1}$ is injective. Let $a \in \t_n$,
$a=a'+a''$, where $a' \in \t_{n-1}$ and $a'' \in 
\lie(t^{1,n}, \dots, t^{n-1,n})$. We denote by $A'$ and $A''$
their images in $\sder_{n}$. Observe that $A'(x_n)=0$ since 
$A'$ is a derivation acting only on generators $x_1, \dots, x_{n-1}$.
It is easy to check that  $A''(x_n)=[x_n, a''(x_1,\dots, x_{n-1})]$,
where $a''(x_1,\dots, x_{n-1})$ is obtained by replacing
the generators $t^{i,n}$ by $x_i$ in $a''(t^{1,n}, \dots, t^{n-1,n})$.
Assuming $A=A'+A''=0$, we have $A(x_n)=0$ which implies $A''(x_n)=0$ and  
$a''=0$. Then, $a=a' \in \t_{n-1}$ and
$A=0$ implies $a=0$ by the induction hypothesis.
\end{proof}

\begin{proposition}
The element $c= \sum_{i<j} t^{i,j}$ belongs to the center of $\sder_n$.
\end{proposition}

\begin{proof}
First, note that $c(x_i)=\sum_{j\neq i}[x_i,x_j]=[x_i,x]$ for 
$x=\sum_{j=1}^n x_j$. Hence,
$c$ is an inner derivation, and for any $a \in \lie_n$ we have $c(a)=[a,x]$.
Let $u=(a_1, \dots, a_k)\in \sder_n$ and compute the $k$th component of the bracket $[c,u]$: 
$$
\begin{array}{lll}
c(a_k) - u(\sum_{i\neq k} x_i) + \sum_{i\neq k}[x_i,a_k]
& = & [a_k, x] + u(x_k) + \sum_{i\neq k}[x_i,a_k] \\
& = & [a_k, x] + [x_k,a_k] + \sum_{i\neq k}[x_i,a_k] \\
& = & [a_k,x] + [x, a_k]=0 .
\end{array}
$$
Here we have used that $u(x)=0$ for $u\in \sder_n$.
\end{proof}

Another family of Lie algebra homomorphisms $\dert_{n-1} \to \dert_n$
is given by {\em coproduct maps}. For $u=(a_1, \dots, a_{n-1}) \in \dert_{n-1}$
we define
$$
\begin{array}{lll}
u^{12,3,\dots,n} & = & (a_1(x_1+x_2,x_3, \dots, x_n), \\
& &                     a_1(x_1+x_2,x_3, \dots, x_n),\\
& &                      a_2(x_1+x_2,x_3, \dots, x_n), \\
& &                      \dots, \\
& &                      a_{n-1}(x_1+x_2,x_3, \dots, x_n)).
\end{array}
$$
Other coproduct maps are obtained by using the action of the permutation
groups on $\dert_{n-1}$ and on $\dert_n$.

\begin{example}
For $n=2$ and $u=(a(x,y),b(x,y))$ we have
$u^{12,3}=(a(x+y,z),a(x+y,z),b(x+y,z))$ and 
$u^{1,23}=(a(x,y+z),b(x,y+z),b(x,y+z))$.
\end{example}

Coproduct maps $\dert_{n-1} \to \dert_n$ are Lie algebra homomorphisms.
Let $u=(a,b) \in \dert_2$ and compute
$u^{12,3}(x+y)=[x+y,a(x+y,z)]$ and $u^{12,3}(z)=[z,b(x+y,z)]$.
Hence, for any $f \in \lie_2$ we obtain $u^{12,3}(f(x+y,z))=(u(f))(x+y,z)$.
For $u=(a_1,b_1), v=(a_2,b_2) \in \dert_2$ we compute
$[u^{12,3},v^{12,3}]=(c_1,c_2,c_3)$ where
$$
\begin{array}{lll}
c_1=c_2 & = & u^{12,3}(a_2(x+y,z))-v^{12,3}(a_1(x+y,z))+[a_1(x+y,z), a_2(x+y,z)] \\
 & = & (u(a_2)-v(a_1)+[a_1,a_2])(x+y,z), \\
c_3 & = & u^{12,3}(b_2(x+y,z))-v^{12,3}(b_1(x+y,z))+[b_1(x+y,z), b_2(x+y,z)] \\
& = & (u(b_2)-v(b_1)+[b_1,b_2])(x+y,z).
\end{array}
$$
Hence, $[u^{12,3},v^{12,3}]=[u,v]^{12,3}$.
Coproduct maps restrict to Lie subalgebras of special derivations.
For $u \in \sder_{n-1}$ and $x=\sum_{i=1}^n x_i$ we compute
$$
u^{12,3,\dots,n}(x) =
[x_1+x_2, a_1(x_1+x_2, \dots, x_n)]+ \dots + [x_n, a_{n-1}(x_1+x_2, \dots, x_n)]
=0
$$
which implies $u^{12,3,\dots,n} \in \sder_n$.

\begin{example}
For $r=(y,0) \in \dert_2$ we have $r^{12,3}=(z,z,0)=r^{1,3}+r^{2,3}$
and $r^{1,23}=(y+z,0,0)=r^{1,2}+r^{1,3}$. Similarly, for $t=(y,x)\in \dert_2$
we have $t^{12,3}=(z,z,x+y)=t^{1,3}+t^{2,3}$ and
$t^{1,23}=(y+z,x,x)=t^{1,2}+t^{1,3}$.
\end{example}

Let $u=(a_1,b_1) \in \sder_2$ and $v=(a_2,b_2) \in \dert_2$.
Then, $[u^{1,2},v^{12,3}]=0$. Indeed, note that
$u^{1,2}$ acts by zero on $\lie(x+y,z)$ and $v^{12,3}$
acts as an inner derivation with generator $a_2(x+y,z)$
on $\lie(x,y)$. We compute 
$$
\begin{array}{lll}
[u^{1,2},v^{12,3}](x) & = & u^{1,2}([x,a_2(x+y,z)])-v^{12,3}([x,a_1(x,y)]) \\
& = & [[x,a_1(x,y)],a_2(x+y,z)] - [[x,a_1(x,y)],a_2(x+y,z)]=0,
\end{array}
$$
and similarly $[u^{1,2},v^{12,3}](y)=0$. Finally,
$[u^{1,2},v^{12,3}](z)=u^{1,2}([z,b_2(x+y,z)])=0$.
In general, for $u\in \sder_n, v\in \dert_{m+1}$ we have
$[u^{1,2,\dots,n}, v^{12\dots n,n+1,\dots,n+m}]=0$.

\subsection{Cohomology}  \label{subsec:dcoh}
We define a differential $\d : \dert_{n} \to \dert_{n+1}$ by formula,
$$
\d u =u^{2,3,\dots,n+1} - u^{12,\dots,(n-1),n} + \dots +
(-1)^{n} u^{1,2,\dots,(n-1)n} + (-1)^{n+1} u^{1,2,\dots,n} .
$$
It is easy to check that $\d$ squares to zero, $\d^2=0$.

\begin{example}
For $u \in \dert_2$ we get
$\d u= u^{2,3}-u^{12,3}+u^{1,23}-u^{1,2}$. For $u\in \dert_3$
we obtain $\d u = u^{2,3,4}-u^{12,3,4}+u^{1,23,4}-u^{1,2,34}+u^{1,2,3}$.
\end{example}

We shall compute the cohomology groups
$$
H^n(\dert, \d)={\rm ker}(\d: \dert_n \to \dert_{n+1})/{\rm im}(\d: \dert_{n-1} \to \dert_n)
$$
for $n=2,3$.

\begin{theorem}   \label{th:hd}
$$
\begin{array}{lll}
H^2(\dert, \d) & =& {\rm ker}(\d: \dert_2 \to \dert_{3}) = \k r \oplus \k t , \\
H^3(\dert, \d) & \cong & \k [(0,[z,x],0)],
\end{array}
$$
where $r=(y,0), t=(y,x)$.
\end{theorem}

\begin{proof}
Since $\dert_1=0$, we have $H^2(\dert, \d)= {\rm ker}(\d: \dert_2 \to \dert_{3})$. 
Let $u=(a,b) \in \dert_2$, and consider $\d u=u^{2,3}-u^{12,3}+u^{1,23}-u^{1,2}$.
Equation $\d u=0$ reads
$$
\begin{array}{lllllllll}
       & - & a(x+y,z) & + & a(x,y+z) & - & a(x,y) & = & 0, \\
a(y,z) & - & a(x+y,z) & + & b(x,y+z) & - & b(x,y) & = & 0, \\
b(y,z) & - & b(x+y,z) & + & b(x,y+z) &   &        & = & 0.
%
\end{array}
$$
Put $x=0$ in the first equation to get $a(y,z)=a(0,y+z)-a(0,y)=\alpha z$.
In the same way, put $z=0$ in the third equation to obtain
$b(x,y)=b(x+y,0)-b(y)=\beta x$. All three equations are satisfied by
$u=(\alpha y, \beta x) = (\alpha - \beta)r + \beta t$ for all $\alpha, \beta \in \k$. Hence,
${\rm ker}(\d: \dert_2 \to \dert_{3}) = \k r \oplus \k t$

In order to compute $H^3(\dert, \d)$ we put $u=(a,b,c) \in \dert_3$ and
write $\d u= u^{2,3,4}-u^{12,3,4}+u^{1,23,4}-u^{1,2,34}+u^{1,2,3}$. Equation
$\d u=0$ yields
$$
\begin{array}{lllllll}
         & -a(x+y,z,w) & +a(x,y+z,w) & -a(x,y,z+w) & +a(x,y,z) & = & 0, \\
a(y,z,w) & -a(x+y,z,w) & +b(x,y+z,w) & -b(x,y,z+w) & +b(x,y,z) & = & 0, \\
b(y,z,w) & -b(x+y,z,w) & +b(x,y+z,w) & -c(x,y,z+w) & +c(x,y,z) & = & 0, \\
c(y,z,w) & -c(x+y,z,w) & +c(x,y+z,w) & -c(x,y,z+w) &           & = & 0, \\
%
\end{array}
$$
Make a substitution $x \mapsto x, y \mapsto -x, z\mapsto x+y, w \mapsto z$
in the first equation to get
$$
a(x,y,z)=a(x,-x,x+y+z)-a(x,-x,x+y)+a(0,x+y,z).
$$
Let $f(x,y)=-a(x,-x,x+y)$ and $k(x,y)=a(0,x,y)-f(x,y)$ to get the following
expression for $a$,
$$
a(x,y,z)=f(x,y)-f(x,y+z)+f(x+y,z) + k(x+y,z).
$$
In the same fashion, putting $x \mapsto y, y\mapsto z+w, z\mapsto -w, w\mapsto w$
in the forth equation gives
$$
c(y,z,w)=c(y+z+w,-w,w)-c(z+w,-w,w)+c(y,z+w,0).
$$
By letting $g(z,w)=-c(z+w,-w,w)$ and $l(z,w)=c(z,w,0)+g(z,w)$ we obtain
$$
c(y,z,w)=-g(y,z+w)+g(y+z,w)-g(z,w) + l(y,z+w).
$$
Consider $\tilde{u}=(\tilde{a}, \tilde{b}, \tilde{c})=u+\d(f,g)$. It satisfies
$\d \tilde{u}=0$ and it has $\tilde{a}(x,y,z)=k(x+y,z)$ and $\tilde{c}(x,y,z)=l(x,y+z)$.
The first equation (for $\tilde{a}$) implies $k(x+y,z)=k(x+y,z+w)$ which forces
$k=0$ (since $\tilde{a}$ does not contain terms linear in $x$).
In the same way, the forth equation yields $l(x+y,z+w)=l(y,z+w)$ which implies $l=0$.
Hence, $\tilde{u}=(0,\tilde{b},0)$. Denote $h(x,y)=\tilde{b}(x,0,y)$ and
first put $y=0$ in the third equation to get $\tilde{b}(x,z,w)=h(x,z+w)-h(x,z)$,
then put $z=0$ to obtain $\tilde{b}(x,y,w)=h(x+y,w)-h(y,w)$.
These two equations imply
$$
h(x,y)-h(x,y+w)+h(x+y,w)-h(y,w)=0,
$$
and, by Theorem \ref{th:hdelta}, $h(x,y)=\gamma [x,y]$ for some $\gamma \in \k$.
This implies $\tilde{b}(x,y,z)=\gamma [x,y+z]- \gamma [x,y]=\gamma[x,z]$.
It is easy to check that $\tilde{u}=(0, \gamma [x,z], 0)$ verifies $\d \tilde{u}=0$.
Finally, in degree two, ${\rm im}(\d: \dert_2 \to \dert_3)$ is spanned by 
$$
\d (\alpha [x,y], \beta [x,y])=(-\alpha [y,z], (\alpha-\beta) [z,x], \beta [x,y]),
$$
and $(0, \gamma [x,z], 0) \notin {\rm im}(\d: \dert_2 \to \dert_3)$ for $\gamma \neq 0$.
\end{proof}

\subsection{Cocycles in $\tr_n$} \label{subsec:cocycle}
The action of $\der_n$ extends from $\lie_n$ to $\Ass_n$
and descends to the graded vector space $\tr_n$. For $u \in \der_n$ and
$a \in \tr_n$ we denote this action by $u \cdot a \in \tr_n$.

\begin{example}
Let $r=(y,0) \in \dert_2$, and $a=\Tr(xy) \in \tr_2$. We compute
$r\cdot a=\Tr(r(x)y +xr(y))=\Tr([x,y]y)=\Tr((xy-yx)y)=0$.
\end{example}

 We shall be
interested in 1-cocycles on the subalgebra $\dert_n$ with values in 
$\tr_n$. That is, we are looking for linear maps $\alpha: \dert_n \to \tr_n$ such that
$$
u\cdot \alpha(v)- v\cdot \alpha(u)-\alpha([u,v])=0
$$
for all $u,v \in \dert_n$. 

\begin{proposition}
For all $k=1, \dots, n$ the map $\alpha: u=(a_1, \dots, a_n) \mapsto \Tr(a_k)$
is a 1-cocycle.
\end{proposition}

\begin{proof}
Note that $\alpha$ vanishes on all elements of degree
greater or equal to two. Hence, $\alpha([u,v])=0$ for
all $u,v \in \dert_n$. Let $u=(a_1, \dots, a_n)$ and $v=(b_1, \dots, b_n)$.
Then, $u \cdot \alpha(v)=u\cdot \Tr(b_k)=\Tr(u(b_k))=0$
since $u(b_k)$ is of degree at least two, and similarly
$v \cdot \alpha(u)=\Tr(v(a_k))=0$. 
\end{proof}

\begin{proposition}
The map $\dv: u=(a_1, \dots, a_n) \mapsto \sum_{k=1}^n \Tr(x_k (\partial_k a_k))$
is a 1-cocycle.
\end{proposition}

\begin{proof}
On the one hand, we get
$$
\begin{array}{lll}
u\cdot \dv(v) - v\cdot \dv(u) & = &
\sum_{k=1}^n \Tr \left( u(x_k (\partial_k b_k)) - v(x_k (\partial_k a_k))  \right) \\
& = & 
\sum_{k=1}^n \Tr ( [x_k,a_k](\partial_k b_k) + x_k u(\partial_k b_k) \\
& - &
[x_k, b_k] (\partial_k a_k)- x_k v(\partial_k a_k)) .
\end{array}
$$
On the other hand, we obtain,
$$
\begin{array}{lll}
\dv([u,v]) & = & \sum_{k=1}^n \Tr ( x_k \partial_k(u(b_k)-v(a_k)+[a_k,b_k]) ) \\
& = & \sum_{k=0}^n \Tr (x_k \partial_k (u(\sum_{i=1}^n (\partial_i b_k) x_i)
 -  v(\sum_{j=1}^n (\partial_j a_k) x_j) +[a_k, b_k]) ) \\
& = & \sum_{k=0}^n \Tr (x_k \partial_k (\sum_{i=1}^n (u(\partial_i b_k) x_i + 
(\partial_i b_k) [x_i, a_i]) \\
& - & \sum_{j=1}^n (v(\partial_j a_k) x_j + (\partial_j a_k) [x_j, a_j]) +
[a_k, b_k]) ) \\
& = & \sum_{k=0}^n \Tr (x_k (u(\partial_k b_k) - (\partial_k b_k) a_k +
\sum_{i=1}^n (\partial_i b_k) x_i (\partial_k a_i) \\
& - & v(\partial_k a_k) + (\partial_k a_k) b_k - \sum_{j=1}^n (\partial_j a_k) x_j
(\partial_k b_j) + a_k (\partial_k b_k) - b_k(\partial_k a_k) )) \\  
& = & \sum_{k=1}^n \Tr ( x_k (u(\partial_k b_k)-(\partial_k b_k) a_k -
v(\partial_k a_k) \\
& + & (\partial_k a_k) b_k + a_k (\partial_k b_k)- b_k(\partial_k a_k) ) ) \\
& = & u\cdot \dv(v) - v\cdot \dv(u) .
\end{array}
$$
proving the cocycle condition. Here we have used the definition of $\partial_k$ 
operators (see equation \eqref{eq:partial}) and the fact that 
$a_k=\sum_{j=1}^n (\partial_j a_k) x_j$ and $b_k=\sum_{i=1}^n (\partial_i b_k) x_i$.
\end{proof}

The {\em divergence cocycle} transforms in a nice way under simplicial
and coproduct maps. For $u=(a_1,\dots,a_n)\in \dert_n$ we have
$\dv(u^{1,2,\dots,n})=\sum_{i=1}^n \Tr(x_i (\partial_i a_i))=
\dv(u)(x_1,\dots, x_n)$. For $\dv(u^{12,\dots,n+1})$ we compute
$$
\begin{array}{lll}
\dv(u^{12,\dots,n+1}) & = & \Tr(x_1 (\partial_1 a_1(x_1+x_2, \dots))+
                          x_2 (\partial_2 a_1(x_1+x_2, \dots)) ) \\
& + & \sum_{k=3}^{n+1} \Tr(x_k (\partial_k a_{k-1}(x_1+x_2, \dots))) \\
& = & \Tr((x_1+x_2)(\partial_1 a_1)(x_1+x_2, \dots) \\
& + & \sum_{k=2}^n x_{k+1} (\partial_k a_k)(x_1+x_2, \dots) ) \\
& = & (\dv(u))(x_1+x_2, x_3,\dots, x_{n+1}).
\end{array}
$$

\begin{proposition}
$\dv(\d u)=\delta(\dv(u)).$
\end{proposition}

\begin{proof}
We compute,
$$
\begin{array}{lll}
\dv(\d u) & = & 
\dv(u^{2,\dots,n+1}) - \dv(u^{12,\dots,n+1}) + \dots + (-1)^{n+1} \dv(u^{1,2,\dots,n}) \\
& = & \dv(u)(x_2,\dots,x_{n+1})-\dv(u)(x_1+x_2,\dots,x_{n+1})+\dots  \\
& + &  (-1)^{n+1}\dv(x_1,\dots,x_{n}) \\
& = & \delta(\dv(u)).
\end{array}
$$
\end{proof}

\section{Kashiwara-Vergne Lie algebras}  \label{sec:KVlie}

\subsection{Definitions}
In this section we introduce a family of subalgebras of $\sder_n$
called {\em Kashiwara-Vergne} Lie algebras. 

\begin{definition}
The Kashiwara-Vergne Lie algebra $\kv_n$ is a Lie subalgebra of special
derivations spanned by elements with vanishing divergence.
\end{definition}

Note that $\kv_n$ is indeed a Lie subalgebra of $\sder_n$. For two 
derivations $u,v \in \kv_n$ the cocycle property for divergence
implies $\dv([u,v])=u\cdot \dv(v)-v\cdot \dv(u)=0$, as required.

\begin{example}
The element $t=(y,x)\in \sder_2$ is contained in $\kv_2$. Indeed,
we have $a(x,y)=y, b(x,y)=x$ and $\partial_x a=\partial_y b=0$
which implies $\dv(t)=0$.
\end{example}

Simplicial and coproduct maps restrict to $\kv_n$ subalgebras.
Indeed, for $u \in \sder_n$ the condition $\dv(u)=0$ implies
$\dv(u^{1,2,\dots,n})=0$ and $\dv(u^{12,3,\dots,n+1})=0$.

\begin{example}
Since $t \in \kv_2$, we have $t^{1,2}, t^{1,3}, t^{2,3} \in \kv_3$
and $[t^{1,3},t^{2,3}]=([y,z],[z,x],[x,y]) \in \kv_3$.
\end{example}

In the case of $n=2$ we introduce an extension of $\kv_2$,
$$
\hkv_2:=\{ u \in \sder_2, \dv(u)\in {\rm ker}(\delta)\} .
$$
Recall that ${\rm ker}(\delta: \tr_2 \to \tr_3)={\rm im}(\delta: \tr_1 \to \tr_2)$.
Hence, for $u\in \hkv_2$ there exists an element $f \in \tr_1$ such that
$\dv(u)=\Tr(f(x)-f(x+y)+f(y))$. By Theorem \ref{th:hdelta}, such an element is unique
if we choose it in the form $f(x)=\sum_{k=2}^\infty f_k x^k$. By abuse of notations
we denote by $f$ the map $f: u \mapsto f$, and by $f_k$ the maps $f_k: u \mapsto f_k$.

The subspace $\hkv_2$ is a Lie subalgebra of $\sder_2$. Indeed,
for two derivations $u,v \in \hkv_2$ we compute
$\dv([u,v])=u \cdot \dv(v) - v \cdot \dv(u)$. We have
$\dv(v)=\delta f=\Tr(f(x)-f(x+y)+f(y))$ with $f\in x^2 \k[[x]]$. 
Note that $u \cdot \Tr(f(x+y))=0$ since $u(x+y)=0$
and $u \cdot \Tr(f(x))= \Tr ([x,a] f'(x)) =\Tr([xf'(x),a])=0$, where $u(x)=[x,a]$. Hence,
$u \cdot \dv(v)=0$, and similarly $v\cdot \dv(u)=0$. In fact, we proved
$[\hkv_2, \hkv_2] \subset \kv_2$. 

\begin{proposition}
Let $u \in \hkv_2$. Then, $f(u)$ is odd, and Taylor coefficients 
$f_k, k=3,5,\dots$ are characters of $\hkv_2$. 
\end{proposition} 

\begin{proof}
Let $u\in \hkv_2$ with divergence 
$\dv(u)=\Tr(f(x)-f(x+y)+f(y))$, where $f(x)=\sum_{k=2}^\infty f_k x^k$.
Note that the coefficient in front of $\tr(xy^{n-1})$ in $\dv(u)$ is equal
to $-nf_n$. 
Since $u=(a,b)\in \hkv_2$, we have $u(x+y)=[x,a]+[y,b]=0$. Consider terms
linear in $x$ in both $a$ and $b$. First, observe that $b$ does not contain terms
of the form $\ad_y^m(x)$ for $m\geq 1$ since $\ad_y^{m+1}(x) \notin {\rm im}(\ad_x)$.
In particular, this applies to all $m$ odd. 
Next, note that $a$ does not contain terms of the form $\ad_y^m(x)$ for $m$ odd
since in this case $[x,\ad_y^m(x)] \notin {\rm im}(\ad_y)$. Hence, $\dv(u)=\Tr(x \partial_x a
+y \partial_y b)$ does not contain terms of the form $\Tr(xy^m)$ for $m$ odd, and 
$f_k=0$ for all $k=m+1$ even.
Finally, Taylor coefficients of $f$ are characters of $\hkv_2$ since 
they vanish on $\kv_2$, and on $[\hkv_2, \hkv_2] \subset \kv_2$. 
\end{proof}

%

\subsection{The Grothendieck-Teichm\"uller Lie algebra}
Recall that the Grothendieck-Teichm\"uller Lie algebra $\grt$ was
defined by Drinfeld \cite{dr} in the following way. It is spanned
by derivations $(0,\psi) \in \dert_2$ which satisfy the following three relations
\begin{equation} \label{eq:grt1}
\psi(x,y)=-\psi(y,x),
\end{equation}
\begin{equation}   \label{eq:grt2}
\psi(x,y)+\psi(y,z)+\psi(z,x)=0
\end{equation}
for $x+y+z=0$ (that is, one can put $z=-x-y$),
\begin{equation}  \label{eq:grt3}
\psi(t^{1,2},t^{2,34})+\psi(t^{12,3},t^{3,4})=
\psi(t^{2,3},t^{3,4})+\psi(t^{1,23},t^{23,4})+\psi(t^{1,2},t^{2,3}),
\end{equation}
where the last equation takes values in the Lie algebra $\t_4$ and
$t^{1,23}=t^{1,2}+t^{1,3}$ {\em etc.} Note that defining equations of $\grt$
have no solutions in degrees one and two.
The Lie bracket induced on solutions of \eqref{eq:grt1},
\eqref{eq:grt2},\eqref{eq:grt3} is called Ihara bracket,
$$
[\psi_1, \psi_2]_{\rm Ih}=(0,\psi_1)(\psi_2)-(0,\psi_2)(\psi_1)+[\psi_1,\psi_2].
$$

\begin{theorem}  \label{th:grt}
The map $\nu: \psi \mapsto (\psi(-x-y,x),\psi(-x-y,y))$ is an
injective Lie algebra homomorphism mapping $\grt$ to $\hkv_2$.
\end{theorem}

We split the proof of Theorem \ref{th:grt} into several steps.

\begin{proposition}  \label{prop:dpsi}
Let $\psi \in \grt$. Then, $\Psi=\nu(\psi)$ verifies
\begin{equation} \label{eq:dpsi}
\d \Psi = \psi(t^{1,2}, t^{2,3}).
\end{equation}
\end{proposition}

We defer the proof of this proposition to Appendix.

\begin{proposition}
${\rm im}(\nu) \subset \hkv_2$.
\end{proposition}

\begin{proof}
Using equation \eqref{eq:dpsi} we compute
$$
\delta(\Psi(x+y))=(\d \Psi)(x+y+z)=\psi(t^{1,2}, t^{2,3})(x+y+z)=0
$$
because $t^{1,2}, t^{2,3} \in \sder_3$. Since $\Psi \in \dert_2$ is of degree at least
three, $\Psi(x+y)$ is of degree at least four, and by Theorem \ref{th:hdelta}
this implies $\Psi(x+y)=0$ and $\Psi \in \sder_2$.

Similarly, we compute
$$
\delta(\dv(\Psi))=\dv(\d \Psi)=\dv( \psi(t^{1,2}, t^{2,3})) =0
$$
since $t^{1,2}, t^{2,3} \in \kv_3$. By Theorem \ref{th:hdelta}, this implies
$\dv(\Psi) \in {\rm im}(\delta)$ and $\Psi \in \hkv_2$.
\end{proof}

\begin{proposition}
$\nu: \grt \to \hkv_2$ is a Lie algebra homomorphism.
\end{proposition}

\begin{proof}
Let $\psi_1,\psi_2 \in \grt$ and compute $(a,b)=[\nu(\psi_1), \nu(\psi_2)]$,
$$
\begin{array}{lll}
a(x,y) & = & \nu(\psi_1)(\psi_2(-x-y,x))-\nu(\psi_2)(\psi_1(-x-y,x) \\
       & + & [\psi_1(-x-y,x),\psi_2(-x-y,x)] \\
 & = & \left( (0,\psi_1)(\psi_2)-(0,\psi_2)(\psi_1)+[\psi_1,\psi_2] \right)(-x-y,x),
\end{array}
$$
where we used that $\nu(\psi_1), \nu(\psi_2) \in \sder_2$. Similarly, we have
$$
\begin{array}{lll}
b(x,y) & = & \nu(\psi_1)(\psi_2(-x-y,y))-\nu(\psi_2)(\psi_1(-x-y,y) \\
       & +  & [\psi_1(-x-y,y),\psi_2(-x-y,y)] \\
 & = & \left( (0,\psi_1)(\psi_2)-(0,\psi_2)(\psi_1)+[\psi_1,\psi_2]\right)(-x-y,y).
\end{array}
$$
In conclusion, $[\nu(\psi_1), \nu(\psi_2)]=\nu([\psi_1, \psi_2]_{\rm Ih})$, as required.
\end{proof}
This observation completes the proof of Theorem \ref{th:grt}.

It is known \cite{Ihara,dr} that there exit elements $\sigma_{2n+1}\in \grt$ of degree
$2n+1$ for all $n=1,2,\dots$ Modulo the double commutator ideal
$[[\lie_2, \lie_2],[\lie_2, \lie_2]]$, $\sigma_{2n+1}$ has the following form,
\begin{equation}  \label{eq:sigma}
\sigma_{2n+1}=\sum_{k=1}^{2n} \, \frac{(2n+1)!}{k! (2n+1-k)!} \,
\ad_x^{k-1} \ad_y^{2n-k} [x,y].
\end{equation}

\begin{proposition}  \label{prop:fsigma}
$f\circ \nu(\sigma_{2n+1})= - x^{2n+1}.$
\end{proposition}

\begin{proof}
Equation \eqref{eq:sigma} implies that the linear in $x$ part of 
$a(x,y)=\sigma(-x-y,x)$ is equal to $(2n+1)\ad_y^{2n}x$, and the linear in
$x$ part of $b(x,y)=\sigma(-x-y,y)$ vanishes. Hence,
the coefficient in front of $\Tr(x y^{2n})$ in $\dv(\nu(\sigma_{2n+1}))$
is equal to $(2n+1)$, and
$$
\dv(\nu(\sigma_{2n+1}))=-\Tr(x^{2n+1}-(x+y)^{2n+1}+y^{2n+1})=-\delta \Tr(x^{2n+1}),
$$
which implies $f(\nu(\sigma_{2n+1}))=-x^{2n+1}$.
\end{proof}

Theorem \ref{th:grt} shows that 
$\hkv_2$ is infinite dimensional, and
Proposition \ref{prop:fsigma} implies that characters $f_k, k=3,5,\dots$
are surjective.
The Lie algebra $\hkv_2$ contains a central one dimensional Lie subalgebra
$\k t$ for $t=(y,x)$, and a Lie subalgebra isomorphic to 
the Lie algebra $\grt$. This observation suggests
the following conjecture on the structure of $\hkv_2$.

\vskip 0.2cm

\textbf{Conjecture.} The Lie algebra $\hkv_2$ is isomorphic to a direct
sum of the Grothendieck-Teichm\"uller Lie algebra $\grt$ and a one dimensional 
Lie algebra with generator in degree one, $\hkv_2 \cong \k t \oplus \grt$.

\vskip 0.2 cm

\begin{remark}
The Deligne-Drinfeld conjecture (see Section 6, \cite{dr}) states that 
$\grt$ is a free Lie algebra with generators $\sigma_{2n+1}$. 
In \cite{racinet}, Racinet introduced a graded Lie algebra
$\dmr_0$ related to combinatorics of multiple zeta values.
A numerical experiment of \cite{ENR} shows that up to degree 19
the Lie algebra $\dmr_0$ is freely generated by $\sigma_{2k+1}$,
and that $\dmr_0 \subset \grt$.
A numerical computation by Albert and the second author \cite{aht}
shows that up to degree 16 the dimensions of graded components
of $\hkv_2$ coincide with those of $\k t \oplus \lie(\sigma_3, \sigma_5, \dots)$
(up to degree 7, the computation has been done by Podkopaeva \cite{mp})
.
Since $\k t \oplus \nu(\grt) \subset \hkv_2$, we conclude that
the Conjecture stated above and the Deligne-Drinfeld conjecture are 
verified up to degree 16.
\end{remark}

\section{The Kashiwara-Vergne problem}  \label{sec:kv}

\subsection{Automorphisms of free Lie algebras}
Recall that one can associate a group $G$ 
to a positively graded Lie algebra $\g=\prod_{k=1}^\infty \g_k$
with all graded components of finite dimension. 
$G$ coincides with $\g$ as a set, and 
the group multiplication is defined by the Campbell-Hausdorff formula. 
If $\g$ is finite dimensional, $G$ is the connected and simply connected 
Lie group with Lie algebra $\g$. Even for $\g$ infinite dimensional we shall 
denote the map identifying $\g$ and $G$ by $\exp: \g \to G$ and its inverse 
by $\ln: G\to \g$. Then, the definition of the group multiplication in $G$
reads: $\exp(u)\exp(v)=\exp(\ch(u,v))$. 

Lie algebras $\dert_n, \sder_n, \kv_n$ and $\hkv_2$ introduced 
in the previous Section are positively graded, and all their graded components
are finite dimensional. Hence, they integrate to groups. We shall denote these
groups by $\Dert_n, \Sder_n, \KV_n$ and $\HKV_2$, respectively. 
The natural actions of $\dert_n, \sder_n, \kv_n$ and $\hkv_2$
on $\lie_n$ and on $\tr_n$ lift to actions of the corresponding
groups given by formula

$$
\exp(u)(a):= \sum_{n=0}^\infty u^n(a),
$$
where $u^n(a)$ is the $n$-tuple action of the derivation $u$ on $a$.
Note that the group $\Dert_n$ consists of automorphisms $g$ of 
$\lie_n$ with the property that for each $i=1,\dots,n$ there is
an inner automorphism $g_i$ such that $g(x_i)=g_i(x_i)$. Furthermore,
the group $\Sder_n$ is a subgroup of $\Dert_n$ singled out by
the condition $g(x)=x$ for $x=\sum_{i=1}^n x_i$.

%

In order to discuss the groups $\KV_n$ and $\HKV_2$ we introduce a
Lie group 1-cocycle $j: \Dert_n \to \tr_n$ which integrates
the Lie algebra 1-cocycle $\dv: \dert_n \to \tr_n$.

\begin{proposition}
There is a unique map $j: \Dert_n \to \tr_n$ which satisfies the group cocycle condition
\begin{equation} \label{eq:j1}
j(gh)=j(g)+g\cdot j(h),
\end{equation}
and has the property
\begin{equation} \label{eq:j2}
\frac{d}{ds} \, j(\exp(su))|_{s=0} = \dv(u).
\end{equation}
\end{proposition}

\begin{proof}
Let $\g$ be a semi-direct sum of $\dert_n$ and $\tr_n$. The cocycle property
of the divergence implies that the map $\dert_n \to \g$ defined by formula
$u\mapsto u+\dv(u)$ is a Lie algebra homomorphism. Define $j(\exp(u))$ by 
formula $\exp(u+\dv(u))=\exp(j(\exp(u)))\exp(u)$. For $g=\exp(u)$ and 
$h=\exp(v)$ we
have
$$
\exp(j(gh))gh=(\exp(j(g))g)(\exp(j(h))h)=\exp(j(g)+g\cdot j(h))gh
$$
which implies \eqref{eq:j1}. 

Equations \eqref{eq:j1} and \eqref{eq:j2} imply the following 
differential equation for $j$:
$$
\frac{d}{ds} \, j(\exp(su)) = \dv(u) + u\cdot j(\exp(su)).
$$
Given the initial condition $j(e)=0$, this equation admits a
unique solution,
$$
j(\exp(u)) = \, \frac{e^u-1}{u} \, \cdot \dv(u)
$$
which proves uniqueness of the cocycle $j$.
\end{proof}

\begin{remark}
Equation \eqref{eq:j1} for $h=g^{-1}$ implies $j(g^{-1})= - g^{-1}\cdot j(g)$.
\end{remark}

\begin{proposition}
The group $\KV_n$ is isomorphic to a subgroup of $\Sder_n$ 
singled out by the condition $j(g)=0$.
\end{proposition}

\begin{proof}
Let $u \in \kv_n$. Then, $\dv(u)=0$ implies $j(\exp(u))=0$ and
$\exp(u) \in \KV_n$. In the other direction, $j(g)=0$ for 
$g=\exp(u)$ implies $\dv(u)=u/(e^u-1)\cdot j(g)=0$, and
$u\in \kv_n$.
\end{proof}

\begin{proposition}   \label{prop:j(g)}
Let $g \in \HKV_2$. Then, $j(g) \in {\rm im}(\delta)$.
\end{proposition}

\begin{proof}
Let $u \in \hkv_2$. Then, $\dv(u)=\Tr(f(x)+f(y)-f(x+y))$ with
$f \in x^2 \k[[x]]$. Note that $u \cdot \Tr(f(x))=u \cdot \Tr(f(y))=0$
since $u$ acts as an inner derivation on $x$ and as a (different) inner
derivation on $y$. Furthermore, $u \cdot \Tr(f(x+y))=0$ because
$u(x+y)=0$. Hence, $u \cdot \dv(u)=0$, and
$j(\exp(u))=(e^u-1)/u \cdot \dv(u)=\dv(u) \in {\rm im}(\delta)$.
\end{proof}

\subsection{Scaling transformations}
For $0 \neq s \in \k$ consider an automorphism $A_s$ of the free Lie
algebra $\lie_n$ such that $A_s: x_i \mapsto sx_i$ for all $i=1,\dots, n$.
We have $A_{s_1} A_{s_2}=A_{s_1+s_2}$, $(A_s)^{-1}=A_{s^{-1}}$,
and $A_1=e$. For example, we compute
$$
A_s(\ch(x,y))=\ch(sx,sy)=s\ch_s(x,y).
$$
Note that for $g \in \Dert_n$ an automorphism $g_s=A_s g A_s^{-1}$ is also
an element of $\Dert_n$. Indeed, $g(x_i)=g_i(x_i)=e^a x_i e^{-a}$, where
$g_i$ is an inner automorphism of $\lie_n$ given by conjugation by $e^a$
for $a \in \lie_n$. Then, 
$$
g_s(x_i)=A_s g A_s^{-1}(x_i)=s^{-1} A_s g(x_i)=e^{A_s(a)} x_i e^{-A_s(a)}
$$
proving $g_s \in \Dert_n$. Moreover, since $a_s=A_s(a)$ is analytic
in $s$ with $a_0=0$, we conclude that $g_s$ is also analytic in $s$ with
$g_0=e$.
We shall denote the derivative of $g_s$ with respect to 
the scaling parameter $s$ by $\dot{g}_s$.

\begin{proposition}
Let $g \in \Dert_n$. Then, $u_s:=\dot{g_s}g_s^{-1}$ has the property 
$u_s=s^{-1} A_s u A_s^{-1}$, where $u=u_1$. 
\end{proposition}

\begin{proof}
Let $l$ be a derivation of $\lie_n$ defined by the property
$l(x_i)=x_i$ for all $i$. We have, $\dot{A}_s A_s^{-1}=s^{-1} l$,
and
$$
u_s=\dot{g}_s g_s^{-1}=s^{-1}(l-g_slg_s^{-1})=s^{-1} A_s(l-glg^{-1})A_s^{-1} .
$$
Hence, $u=u_1=l-glg^{-1}$ and $u_s=s^{-1} A_s u A_s^{-1}$ as required.
\end{proof}

Note that $u_s=s^{-1}(a_1(sx_1,sx_2,\dots),\dots)$ is analytic in $s$ with 
$u_0$ given by the degree one component of $u$. For $g \in \Dert_n$ we denote
by $\kappa_s: \Dert_n \to \dert_n$ the map $\kappa_s: g \mapsto
u_s=s^{-1} A_s(l-glg^{-1})A_s^{-1}$, and we put $\kappa=\kappa_1$. 
Similarly, let $u \in \dert_n$, set $u_s=s^{-1} A_s u A_s^{-1}$ and denote
by $E_s: \dert_n \to \Dert_n$ the map $E_s: u \mapsto g_s$ defined as a unique 
solution of the ordinary differential equation $\dot{g}_s g_s^{-1}=u_s$
with initial condition $g_0=e$. We denote $E=E_1$.

\begin{proposition}
The maps $E$ and $\kappa$ are inverse to each other.
\end{proposition}

\begin{proof}
Let $g \in \Dert_n$ and consider $u=\kappa(g)$. Then, $u_s=s^{-1} A_s u A_s^{-1}=\kappa_s(g)$
and $g_s=A_s g A_s^{-1}$ is a solution of the ordinary differential equation (ODE)
$\dot{g}_s=u_sg_s$ with initial condition $g_0=e$. But so does $E_s(u)$. Hence,
by the uniqueness property for solutions of ODEs, we have $g=E(u)=E(\kappa(g))$. 
In the other direction,
let $u \in \dert_n$ and consider $g=E(u)$. Then, $g_s=A_s g A_s^{-1}=E_s(u)$ and
$\kappa_s(g)=\dot{g}_sg_s^{-1}=u_s$. Hence, $\kappa(E(u))=u$ as required.
\end{proof}

Automorphisms $A_s$ extend from $\lie_n$ to $\Ass_n$ and to $\tr_n$.
Note that for $u \in \dert_n$ and $u_s=s^{-1} A_s u A_s^{-1}$ we have
$\dv(u_s)=s^{-1} A_s \cdot \dv(u)$. Similarly, for $g \in \Dert_n$
and $g_s = A_s g A_s^{-1}$ we obtain $j(g_s)=A_s \cdot j(g)$.

\begin{proposition}
Let $g \in \Dert_n$ and  $u=\kappa(g)$. Then,
\begin{equation}  \label{eq:jscaling}
\frac{d j(g_s)}{ds}=u_s \cdot j(g_s) + \dv(u_s) . 
\end{equation}
\end{proposition}

\begin{proof}
We compute
$$
j(g_q)=j(g_q g_s^{-1} g_s) = j(g_q g_s^{-1}) + (g_q g_s^{-1})\cdot j(g_s).
$$
Taking a derivative with respect to $q$ and putting $q=s$ yields the equation 
\eqref{eq:jscaling}, as required.
\end{proof}

For $g=E(u)$, equation \eqref{eq:jscaling} at $s=1$ implies the following
relation between $j(g)$ and $\dv(u)$: $l \cdot j(g)= u\cdot j(g) + \dv(u)$.
By using equation $u=l-glg^{-1}$ we obtain $glg^{-1} \cdot j(g)=\dv(u)$.

\subsection{The generalized Kashiwara-Vergne problem}
The generalized Kashiwara-Vergne (KV) problem is the following question:

\vskip 0.2cm
\textbf{Generalized KV problem:} Find an element $F\in \Dert_2$ 
with the properties
\begin{equation}  \label{eq:newkv1}
F(x+y)=\ch(x,y),
\end{equation} 
and
\begin{equation}  \label{eq:newkv2}
j(F) \in {\rm im}(\tilde{\delta}) .
\end{equation}
\vskip 0.2cm
We shall denote the set of solutions of the generalized 
KV problem by $\Sol(\HKV)$. For any $s\in \k$ one can introduce
rescaled versions of equations \eqref{eq:newkv1} and \eqref{eq:newkv2} as
$F(x+y)=\ch_s(x,y)$ and $j(F) \in {\rm im}(\tidelta_s)$. 
We shall denote the corresponding set of solutions by 
$\Sol_s(\HKV)$. For $s=0$, $\Sol_0(\HKV)=\HKV_2$. For
all $s\neq 0$, $\Sol_s(\HKV) \cong \Sol(\HKV)$ with
isomorphism given by the scaling transformation 
$F \mapsto F_s=A_s F A_s^{-1}$.

\begin{proposition}
Let $F\in \Sol(\HKV)$ and $a \in \tr_1$. Then,
$\tilde{\delta} a=F\cdot (\delta a)$.
\end{proposition}

\begin{proof}
We have, $a=\Tr(f(x))$ for some formal power series $f$.
We compute 
$$
\begin{array}{lll}
F\cdot (\delta a) & = & F\cdot \Tr(f(x)-f(x+y)+f(y)) \\
& = & \Tr(f(x)-f(\ch(x,y))+f(y))=\tilde{\delta} a.
\end{array}
$$
Here we used that 
$F\cdot \Tr(f(x))=\Tr(f(x))$ and  $F\cdot \Tr(f(y))=\Tr(f(y))$ 
since $F$ acts as an inner automorphism on $x$ and as a (different)
inner automorphism on $y$. We also used that $F\cdot \Tr(f(x+y))=
\Tr(f(\ch(x,y)))$ because  $F(x+y)=\ch(x,y)$.
\end{proof}

The fact that $\Sol(\HKV)$ is non empty has been proved in \cite{am}.
We shall give an alternative proof in the end of the paper.
In order to preserve the logic of the presentation, we shall
not be using the existence of solutions of the KV problem until
we prove it.

\begin{theorem}  \label{th:solhkv}
Assume that $\Sol(\HKV)$ is nonempty. Then, the group $\HKV_2$ acts 
on $\Sol(\HKV)$ by multiplications on the right. This action is free 
and transitive.
\end{theorem}

\begin{proof}
Let $F \in \Sol(\HKV)$ and $g\in \HKV_2$. Then,
$(Fg)(x+y)=F(g(x+y))=F(x+y)=\ch(x,y)$ and
$j(Fg)=j(F)+F\cdot j(g)$. Note that $j(F) \in {\rm im}(\tilde{\delta})$
and, by Proposition \ref{prop:j(g)}, $j(g) \in {\rm im}(\delta)$. 
Hence, $F\cdot  j(g) \in {\rm im}(\tilde{\delta})$
and $j(Fg) \in {\rm im}(\tilde{\delta})$. In conclusion, $\HKV_2$ acts on
the set $\Sol(\HKV)$ by right multiplications. This action is free since
the multiplication on the right is.

Let $F_1, F_2 \in \Sol(\HKV)$ and put $g=F_1^{-1}F_2$. We have,
$g(x+y)=F_1^{-1}(F_2(x+y))=F_1^{-1}(\ch(x,y))=x+y$
and $j(g)=j(F_1^{-1})+F_1^{-1}\cdot j(F_2)=
F_1^{-1} \cdot (j(F_2)-j(F_1))$. Since $j(F_1),j(F_2) \in {\rm im}(\tilde{\delta})$,
we have $F_1^{-1} \cdot (j(F_2)-j(F_1)) \in {\rm im}(\delta)$ and $g \in \HKV_2$.
Hence, the action of $\HKV_2$ on $\Sol(\HKV)$ is transitive.
\end{proof}

The Kashiwara-Vergne problem was stated in \cite{kv} in somewhat
different terms. We shall now establish a relation between our approach and the original 
formulation of the KV problem (KV conjecture).

\begin{theorem}   \label{th:old=new}
An element $F\in \Dert_2$ is a solution of the generalized KV problem
if and only if $u=\kappa(F)=(A(x,y),B(x,y))$ satisfies the following two properties,
\begin{equation}  \label{eq:oldkv1}
x+y-\ch(y,x)=(1-\exp(-\ad_x))A(x,y)+(\exp(\ad_y)-1)B(x,y),
\end{equation}
and
\begin{equation}  \label{eq:oldkv2}
\dv(u) \in {\rm im}(\tilde{\delta}).
\end{equation}
\end{theorem}

\begin{proof}
First, we show that equation $F(x+y)=\ch(x,y)$ is equivalent to
equation $(d/ds - u_s) \ch_s(x,y)=0$. Indeed, we have
$$
F_s(x+y)=A_s F A_s^{-1}(x+y)=s^{-1} A_sF(x+y)=s^{-1}A_s\ch(x,y)=\ch_s(x,y)
$$
and
$$
u_s (\ch_s(x,y))=\dot{F}_s F_s^{-1} (\ch_s(x,y))=
\dot{F}_s (x+y)=\frac{d}{ds} \, \left( F_s(x+y) \right)= \frac{d \ch_s(x,y)}{ds} .
$$
In the other direction,
$$
\frac{d}{ds} \, F_s^{-1}(\ch_s(x,y)) = F_s^{-1} \left( \frac{d}{ds} - u_s \right)\ch_s(x,y) =0
$$
implies that $F_s^{-1}(\ch_s(x,y))$ is independent of $s$, and 
comparison with the value at $s=0$ gives
$F_s^{-1}(\ch_s(x,y))=x+y$ or $F_s(x+y)=\ch_s(x,y)$.

A straightforward calculation (see Lemma 3.2 of \cite{kv})
shows that equation $(d/ds - u_s) \ch_s(x,y)=0$ 
is equivalent to \eqref{eq:oldkv1}.

Finally, we compare equations \eqref{eq:newkv2} and \eqref{eq:oldkv2}.
Let $F\in \Sol(\HKV)$, $j(F)=\tidelta(\Tr(f(x)))$. We compute,
$$
\begin{array}{lll}
\dv(u)=FlF^{-1}\cdot j(F) & = & FlF^{-1} \cdot \Tr(f(x)-f(\ch(x,y))+f(y)) \\
& = & Fl\cdot \Tr(f(x)-f(x+y)+f(y)) \\
& = & F\cdot \Tr(\phi(x)-\phi(x+y)+\phi(y)) \\
& = & \Tr(\phi(x)-\phi(\ch(x,y))+\phi(y)) \in {\rm im}(\tidelta),
\end{array}
$$
where $\phi=xf'(x)$ results from the action of the derivation 
$l: x^n \mapsto n x^n$. In the other direction, assume $\dv(u) \in {\rm im}(\tidelta)$.
Then, for $u_s=s^{-1} A_s u A_s^{-1}$ we have $\dv(u_s) \in {\rm im}(\tidelta_s)$.
Equation $(d/ds-u_s) j(F_s)=\dv(u_s)$ implies $d/ds (F_s^{-1} \cdot j(F_s))= 
F_s^{-1} \cdot \dv(u_s) \in {\rm im}(\delta)$. Hence, 
$ F_s^{-1} \cdot j(F_s) \in {\rm im}(\delta)$ and $j(F_s) \in {\rm im}(\tidelta_s)$.
\end{proof}

\begin{remark}
Let $\g$ be a finite dimensional Lie algebra over $\k$. Then,
$A, B \in \lie_2$ define a pair of formal power series 
on $\g \times \g$ with values in $\g$ which satisfy equation \eqref{eq:oldkv1}.
By applying the adjoint representation to the equation $\dv(u)=\tilde{\delta}(\phi)$
we obtain an equality in formal power series on $\g \times \g$ with values in $\k$,
\begin{equation} \label{eq:oldkv2'}
{\rm Tr}(\ad_x \circ d_x A + \ad_y \circ d_y B)=
{\rm Tr}(\phi(x)+\phi(y)-\phi(\ch(x,y))).
\end{equation}
Here $(d_x A)(z)=d A(x+tz,y)/dt |_{t=0}$ and 
$(d_y B)(z)=d B(x,y+tz)/dt |_{t=0}$. Indeed, for $A\in \lie_2$ consider 
$U(x,y,z)=d A(x+tz,y)/dt |_{t=0} \in \lie_3$. It has the form
$U=\ad_a(z)$ for some $a \in \Ass_2$. We compute (see equation \eqref{eq:partial}),
$$
a=\partial_z U(x,y,z)=\left( \frac{d}{dt} \, \partial_z A(x+tz,y) \right)|_{t=0}=
\partial_x A
$$
showing $\ad(\partial_x A)=d_x A$. Similarly, $\ad(\partial_y B)=d_y B$.
\end{remark}

\section{Duflo functions}  \label{sec:duflo}

Let $F \in \Sol(\HKV)$. Then, $j(F)=\Tr(f(x)-f(\ch(x,y))+f(y))$, 
and $\dv(\kappa(F))=\Tr(\phi(x)-\phi(\ch(x,y))+\phi(y))$
for $f, \phi \in x^2\k[[x]]$. We shall call $f(x)$ a Duflo function of $F$.
In this Section, we describe the set of formal power series which may arise
as Duflo functions associated to solutions of the KV problem.

\begin{proposition}  \label{prop:bernoulli}
Let $u \in \dert_2$ and assume that it satisfies equations \eqref{eq:oldkv1} and
\eqref{eq:oldkv2} with $\dv(u)=\tidelta(\Tr(\phi(x)))$. Then, the even 
part of the formal power series $\phi$ is given by the following formula
$$
\phi_{even}(x)=\frac{1}{2} \, (\phi(x)+\phi(-x)) =
\frac{1}{2} \, \sum_{k=2}^\infty \frac{B_n}{n!}\, x^n = \frac{1}{2}
\left( \frac{x}{e^x-1} - 1 + \frac{x}{2} \right) \, ,
$$
where $B_n$ are Bernoulli numbers.
\end{proposition}

\begin{proof}
We follow \cite{ap} (see Remark 4.3). Write
$A(x,y)=\alpha(\ad_x) y + \dots, B(x,y)=bx + \beta(\ad_x)y + \dots$,
where $b \in \k, \alpha, \beta \in \k[[x]]$, and $\dots$ stand for 
the terms containing at least two $y$'s.
Replace $y \mapsto sy$ in equation \eqref{eq:oldkv1}, and compute the 
first and second derivatives in $s$ at $s=0$. The first derivative yields
$$
y - \, \frac{\ad_x}{e^{\ad_x} -1 } \, y = (1-e^{-\ad_x}) \alpha(\ad_x) y
- b [x,y],
$$
and we obtain
$$
\alpha(t)=b \, \frac{t}{1-e^{-t}} \, - \, \frac{t}{(e^t-1)(1-e^{-t})} + \, \frac{1}{1-e^{-t}} .
$$
Note that elements of $\lie_2$ 
quadratic in the generator $y$ are 
in bijection with skew-symmetric formal power series in two variables,
$$ 
a(u,v)=\sum_{i,j=0}^\infty a_{i,j} u^i v^j \mapsto
\sum_{i,j=0}^\infty a_{i,j} [\ad_x^i y, \ad_x^j y]
$$
The second derivative of \eqref{eq:oldkv1} gives the following equality 
in formal power series,
$$
\frac{1}{2} \, \frac{(u+v)(e^u-e^v) - (u-v)(e^{u+v}-1)}{(e^{u+v}-1)(e^u-1)(e^v-1)} =
(1-e^{-(u+v)}) a_2(u,v) + \frac{b}{2} (u-v) + (\beta(v)-\beta(u)),
$$
where the left hand side corresponds to the second derivative of the
Campbell-Hausdorff series $-\ch(sy,x)$, and
$a_2(u,v)$ represents the second derivative of $A(x,sy)$. By putting
$v=-u$ in the last equation we obtain,
$$
\beta_{odd}(t)= \frac{b}{2} \, t - \frac{1}{2} \, \frac{t}{(e^t-1)(1-e^{-t})}
+ \frac{1}{4} \, \frac{e^t+1}{e^t-1} \, .
$$
Here $\beta_{odd}(t)=(\beta(t)-\beta(-t))/2$.

Finally, consider equation \eqref{eq:oldkv2'} and compute the contribution
linear in $y$ (that is, of the form  $\Tr(f(x) y)$) on the left hand side and 
on the right hand side.
Since we only control the odd part of the function $\beta(t)$, we obtain
an equation in odd formal power series,
$$
\beta_{odd}(t)-\alpha_{odd}(t)=-(\phi'(t))_{odd}=-(\phi_{even})'(t)
$$
which implies
$$
\phi_{even}(t)= \frac{1}{2} \left(\frac{t}{e^t-1} -1 + \frac{t}{2}\right),
$$
as required.
\end{proof}

\begin{proposition}  \label{prop:oddpowers}
Let $F\in \Sol(\HKV)$ and $f \in x^2\k[[x]]$ such that 
$j(F)=\tidelta(\Tr(f(x)))$. Then, the even part of $f(x)$ 
coincides with the function $f_{even}(x)=\frac{1}{2}\, \ln(e^{x/2}-e^{-x/2})/x)$, 
and for every odd formal power series 
$f_{odd}(x)=\sum_{k=1}^\infty f_{2k+1} x^{2k+1}$ there is an element
$F \in \Sol(\HKV)$ such that $j(F)=\tidelta(\Tr(f_{even}(x)+f_{odd}(x)))$.
\end{proposition}

\begin{proof}
Let $f$ and $\phi$ be the power series in 
$j(F)=\tidelta(\Tr(f(x)))$ and $\dv(u)=\tidelta(\Tr(\phi(x)))$
for $u=\kappa(F)$. Then, we have (see the proof of Theorem \ref{th:old=new})
$\phi(s)=sf'(s)$. By Proposition  \ref{prop:bernoulli},
we obtain 
$$
f_{even}=\int \, \frac{\phi_{even}(s)}{s} \, ds = 
\frac{1}{2} \, \sum_{k=2}^\infty \, \frac{B_k}{k \cdot k!} \, s^k
=\frac{1}{2} \, \ln\left(\frac{e^{s/2}-e^{-s/2}}{s}\right) .
$$

Let $F \in \Sol(\HKV)$ with $j(F)=\tidelta(\Tr(f(x)))$, and 
$g \in \HKV_2$ with $j(g)=\delta(\Tr(h(x)))$. Then, $Fg \in \Sol(\HKV)$
and
$$
j(Fg)=j(F) + F\cdot j(g)=\tidelta(\Tr(f(x)+h(x))).
$$
Put $g=\exp(u)$ for $u\in \hkv_2$, and compute
$j(g)=(e^u-1)/u \cdot \dv(u)=\dv(u)$. By choosing
$u=-\sum_{k=1}^\infty h_{2k+1} \nu(\sigma_{2k+1})$ we obtain
$j(g)=\dv(u)=\delta(\Tr(h(x)))$ for $h(x)= \sum_{k=1}^\infty h_{2k+1} x^{2k+1}$.
Hence, by an appropriate choice of $g \in \HKV_2$, 
one can make the odd part of the linear combination
$f(x)+h(x)$ equal to any given odd power series without linear term.  
\end{proof}

\begin{remark}
The group $\HKV_2$ acts on $\Sol(\HKV)$, and this action descends
to the space of formal power series $x^2 \k[[x]]$ along the map 
$f: \Sol(\HKV) \to x^2 \k[[x]]$.
In Proposition \ref{prop:oddpowers} we have used this action to change
the odd part of $f(F)$. Previously, this action (for the
Grothendieck-Teichm\"uller subgroup $\GRT \subset \HKV_2$)
on the Duflo functions has been described in \cite{kont'} (see Theorem 7).
\end{remark}

\begin{proposition}
Let $F=\exp(u) \in \Sol(\HKV)$ with $u=(a,b) \in \dert_2$ such that
$$
\begin{array}{lll}
a(x,y) & = & a_0 y + \alpha(\ad_y) x + \dots \\
b(x,y) & = & b_0 x + \beta(\ad_y) x+ \dots  \, ,
\end{array}
$$
where $a_0,b_0 \in \k$, $\alpha,\beta \in s\k[[s]]$, and $\dots$ stand for
terms which contain at least two~$x$. Then, the Duflo function associated
to $F$ satisfies equation $f'=\beta-\alpha$.
\end{proposition}

\begin{proof}  \label{prop:extra1}
Consider the part of $j(F)=\Tr(f(x)-f(\ch(x,y))+f(y))$ linear in the 
generator $x$. On the one hand, we have
$$
j(F)_{x-lin}=\Tr(f(x)-f(\ch(x,y))+f(y))_{x-lin}= -\Tr( f'(y) x).
$$
On the other hand, we obtain
$$
j(F)_{x-lin}=\left( \frac{e^u -1}{u} \cdot \dv(u) \right)_{x-lin}=
\dv(u)_{x-lin} .
$$
Here we used the fact that linear in $x$ terms cannot arise under the
action of elements of $\dert_2$ on $\tr_2$. Indeed, such a term would
be of the form $\Tr (h(y) [x,y])$ for some formal power series $h$, and
$\Tr (h(y) [x,y])=\Tr(h(y)yx-h(y)xy)=0$. 

Finally, we compute
$$
\dv(u)_{x-lin}=\Tr(x(\partial_x a)+y(\partial_y b))_{x-lin}=
\Tr(x\alpha(y) - \beta(y) x)=\Tr((\alpha(y)-\beta(y))x).
$$
Comparison with the first equation yields $f'(y)=\beta(y)-\alpha(y)$,
as required.
\end{proof}

In the original formulation of the Kashiwara-Vergne problem 
the Duflo function $f$ was assumed to be even.

\vskip 0.2cm
\textbf{KV problem:} Find an element $F\in \Dert_2$ such that
$F(x+y)=\ch(x,y)$ and 
$j(F)=\frac{1}{2} \, \sum_{k=2}^\infty \frac{B_k x^k}{k\cdot k!}=
\frac{1}{2} \,\ln((e^{x/2}-e^{-x/2})/x)$.
\vskip 0.2cm

We shall denote the set of solutions of the KV problem by $\Sol(\KV)$.
Note that the KV problem is equivalent to finding an element $u=(A,B) \in \dert_2$
which satisfies equation \eqref{eq:oldkv1} and the identity 
$\dv(u)=\tilde{\delta}\left( \frac{1}{2} \, \Tr \sum_{k=2}^\infty \frac{B_k x^k}{k!}\right)$.

\begin{remark}
The group $\KV_2$ acts on $\Sol(\KV)$ by right multiplications.
This action is free and transitive. The proof of this statement
is completely analogous to the proof of Theorem \ref{th:solhkv}.
\end{remark}

\section{Pentagon equation}  \label{sec:pent}

In this Section we establish a relation between the Kashiwara-Vergne
problem and the pentagon equation introduced in \cite{dr}.
Let $\Phi \in \Dert_3$. We say that $\Phi$ satisfies 
the pentagon equation if
\begin{equation} \label{eq:pent}
\Phi^{12,3,4} \Phi^{1,2,34}=\Phi^{1,2,3}\Phi^{1,23,4}\Phi^{2,3,4} .
\end{equation}

\begin{proposition}  
Let $F \in \Sol(\HKV)$. Then,
\begin{equation}  \label{eq:phiffff}
\Phi=(F^{12,3})^{-1}(F^{1,2})^{-1}F^{2,3}F^{1,23}
\end{equation}
is an element of $\KV_3$, and it satisfies the pentagon equation.
\end{proposition}

\begin{proof}
First, we compute
$$
\begin{array}{lll}
\Phi(x+y+z) & = & (F^{12,3})^{-1}(F^{1,2})^{-1}F^{2,3}F^{1,23} (x+y+z) \\
& = & (F^{12,3})^{-1}(F^{1,2})^{-1}F^{2,3}(\ch(x,y+z)) \\
& = & (F^{12,3})^{-1}(F^{1,2})^{-1} (\ch(x,\ch(y,z)) ) \\
& = & (F^{12,3})^{-1} (\ch(x+y,z)) \\
& = & x+y+z .
\end{array}
$$
Hence, $\Phi \in \Sder_3$. Next, we rewrite the defining equation 
for $\Phi$ as $F^{1,2}F^{12,3} \Phi=F^{2,3}F^{1,23}$ and apply 
the cocycle $j$ to both sides to get
$$
j(F^{1,2})+F^{1,2}\cdot j(F^{12,3}) + (F^{1,2}F^{12,3})\cdot j(\Phi)
=j(F^{2,3})+F^{2,3}\cdot j(F^{1,23}).
$$
Since $j(F)=\Tr(f(x)-f(\ch(x,y))+f(y))$, we have
$$
\begin{array}{lll}
j(F^{1,2})+F^{1,2}\cdot j(F^{12,3}) & = & \Tr(f(x)+f(y)-f(\ch(x,y))) \\
& +& F^{1,2} \cdot \Tr(f(x+y)-f(\ch(x+y),z)+f(z)) \\
& = & \Tr(f(x)+f(y)+f(z) - f(\ch(\ch(x,y),z)) )
\end{array}
$$
Similarly, we obtain
$$
\begin{array}{lll}
j(F^{2,3})+F^{2,3}\cdot j(F^{1,23}) & = &
\Tr(f(y)-f(\ch(y,z))+f(z) ) \\
& + & F^{2,3} \cdot \Tr(f(x)-f(\ch(x,y+z))+f(y+z) ) \\
& = & \Tr(f(x)+f(y)+f(z)-f(\ch(x,\ch(y,z))) ).
\end{array}
$$
We conclude $(F^{1,2}F^{12,3})\cdot j(\Phi)=0$, $j(\Phi)=0$
and $\Phi \in \KV_3$.

The pentagon equation is satisfied by substituting the expression
for $\Phi$ into the equation, and by using that for $\Phi \in \KV_3 \subset \Sder_3$
we have
$F^{123,4}\Phi^{1,2,3}=\Phi^{1,2,3}F^{123,4}$ and $F^{1,234}\Phi^{2,3,4}=
\Phi^{2,3,4}F^{1,234}$.
\end{proof}

Let $F_1 \in \Sol(\HKV)$ and $\Phi_1$ be the corresponding solution of 
the pentagon equation. Consider another element $F_2\in \Sol(\HKV)$.
By Theorem \ref{th:solhkv}, $F_2=F_1g$ for some $g \in \HKV_2$. The corresponding
solution of the pentagon equation reads
\begin{equation} \label{eq:twist}
\begin{array}{lll}
\Phi_2 & = & (F_2^{12,3})^{-1}(F_2^{1,2})^{-1}F_2^{2,3}F_2^{1,23}\\
& = & (g^{12,3})^{-1}(F_1^{12,3})^{-1}(g^{1,2})^{-1}(F_1^{1,2})^{-1}
F_1^{2,3}g^{2,3}F_1^{1.23}g^{1,23} \\
& = & (g^{12,3})^{-1}(g^{1,2})^{-1}\Phi_1 g^{2,3}g^{1,23} .
\end{array}
\end{equation}
Equation \eqref{eq:twist} defines an action of $\HKV_2$ 
on solutions of the pentagon equation with values in $\KV_3$. 
Actions of this type are called {\em Drinfeld twists}.

\begin{proposition} \label{prop:F1F2}
Let $F_1, F_2 \in \Sol(\HKV)$ and assume that they give rise to the 
same solution $\Phi$ of the pentagon equation. Then, $F_2=F_1 \exp(\lambda t)$
for some $\lambda \in \k$. 
\end{proposition}

\begin{proof}
First, note that for $g=\exp(\lambda t)$ we have for all $\Phi \in \KV_3$
$$
(g^{12,3})^{-1}(g^{1,2})^{-1}\Phi g^{2,3}g^{1,23} = 
e^{-\lambda c}\Phi e^{\lambda c}= \Phi,
$$
where $c=t^{1,2}+t^{1,3}+t^{2,3}$ is a central element in $\sder_3$ and in
$\kv_3$. 

The degree one component of $\hkv_2$ is spanned by $t$, and $t$ is central in 
$\hkv_2$. Hence, one can represent $g=F_1^{-1}F_2$ in the form
$g=\exp(\lambda t) \exp(u)$, where $u=\sum_{k=2}^\infty u_k \in \hkv_2$. 
Let $\Phi$ be a solution of the pentagon equation which corresponds
to both $F_1$ and $F_2$. Let $k_0$ be the lowest degree  such that 
$u_{k_0} \neq 0$. Then, equation
$\Phi=(g_2^{12,3})^{-1}(g_2^{1,2})^{-1}\Phi g^{2,3}g^{1,23}$
implies $\d u_{k_0}=0$, and by Theorem \ref{th:hd} we have $u_{k_0}=0$
which implies $u=0$ and $g=\exp(\lambda t)$, as required.
\end{proof}

\begin{proposition}
Let $\Phi=\exp(\phi) \in \Dert_2$ be a solution of the pentagon equation,
where $\phi=\sum_{k=1}^\infty \phi_k$ with $\phi_k \in \dert_3$ homogeneous
of degree $k$. Then, $\phi_1=0$ and $\phi_2=(\alpha [y,z], \beta [z,x], \gamma [x,y])$.
\end{proposition}

\begin{proof}
The degree one component of the pentagon equation reads $\d \phi_1=0$.
Since the degree one component of $H^3(\dert, \d)$ vanishes, we have
$\phi_1=\d f$ for a degree one element $f \in \dert_2$. However, the degree
one component of $\dert_2$ is spanned by $r=(0,x)$ and $t=(y,x)$, and both
$r$ and $t$ are in the kernel of $\d$. Hence, $\phi_1=0$. This implies
that the degree two component of the pentagon equation is of the form,
$\d \phi_2=0$. Then (see the proof of Theorem \ref{th:hd}), $\phi_2$ is expressed as
$(\alpha [y,z], \beta [z,x], \gamma [x,y])$ for some $\alpha, \beta, \gamma \in \k$.
\end{proof}

Note that $H^3(\dert, \d)$ is one-dimensional, and the cohomology lies in degree
two. One can choose the isomorphism $H^3(\dert, \d) \cong \k$ in such a way that
it is represented by the map $\pi: \phi_2=(\alpha [y,z], \beta [z,x], \gamma [x,y])
\mapsto \alpha + \beta + \gamma$.

\begin{proposition} \label{prop:kv=>pent}
Let $F=\exp(u)\exp(sr/2)\exp(\alpha t) \in \Dert_2$, where $u$ is
an element of $\dert_2$ of degree greater of equal to two. Assume that 
the expression 
$\Phi=(F^{12,3})^{-1}(F^{1,2})^{-1}F^{2,3}F^{1,23}$
is an element of $\KV_3$, and denote $\pi(\phi_2)=\lambda$. Then,
$\lambda=s^2/8$ and $F \in \Sol_{s}(\HKV)$.
\end{proposition}

\begin{proof}
Note that the degree two component of $\phi=\ln(\Phi)$ is given by
$$
\phi_2=\d u_2 + \frac{s^2}{8}([r^{2,3},r^{1,23}]+[r^{12,3},r^{1,2}])=
\d u_2 + \frac{s^2}{8} [r^{2,3},r^{1,2}]=\d u_2 + \frac{s^2}{8}([y,z],0,0).
$$
Here we used the classical Yang-Baxter equation of Proposition
\ref{prop:cybe}. In conclusion, $\lambda=\pi(\phi_2)=s^2/8$.

Denote $\chi(x,y)=F(x+y)=x+y+\frac{s}{2}[x,y]+\dots$, where $\dots$
stand for elements of degree greater or equal to three. Since 
$\Phi(x+y+z)=x+y+z$, we have
$$
\chi(x,\chi(y,z))=F^{2,3}F^{1,23}(x+y+z)=F^{1,2}F^{12,3}(x+y+z)=\chi(\chi(x,y),z).
$$
By Proposition \ref{prop:chunique}, this implies $\chi(x,y)=\ch_{s}(x,y)$. 
Denote $b(x,y)=j(F) \in \tr_2$. By applying $j$ to the equality
$F^{2,3}F^{1,23}=F^{1,2}F^{12,3}\Phi$ we obtain,
$$
b(y,z)+F^{2,3} \cdot b(x,y+z)=b(x,y)+F^{1,2} \cdot b(x+y,z).
$$
Equivalently, $\tilde{\delta}_{s}(b)=0$ which implies, 
by Proposition \ref{prop:htidelta}, $b \in {\rm im}(\tilde{\delta}_{s})$
and $F \in \Sol_{s}(\HKV)$.
\end{proof}

\begin{theorem}  \label{th:pent=>kv}
Let $\Phi \in \KV_3$ be a solution of the pentagon equation
with $\pi(\phi_2)=\lambda$ and let $s \in \k$ be a square root of $8\lambda$,
$s^2/8=\lambda$. Then, there is a unique element $F \in \Sol_s(\HKV)$
such that $F=\exp(u)\exp(sr/2) \in \Dert_2$, where $u$ is 
an element of $\dert_2$ of degree greater of equal to two, and 
$\Phi=(F^{12,3})^{-1}(F^{1,2})^{-1}F^{2,3}F^{1,23}$.
\end{theorem}

\begin{proof}
Our task is to find $f = \sum_{k=1}^\infty  f_k \in \dert_2$ with the degree
one component $f_1=sr/2$ such that $F=\exp(f)$ solves equation
$\Phi=(F^{12,3})^{-1}(F^{1,2})^{-1}F^{2,3}F^{1,23}$. In degree 
two, it implies,
$$
\d f_2 + \frac{s^2}{8}([y,z],0,0) = \phi_2.
$$
Recall that $\d \phi_2=0$ and $\pi(\phi_2)=\lambda=s^2/8$. Hence,
this equation admits a solution, and it is unique since $\d : \dert_2 \to \dert_3$
has no kernel in degrees greater than one.

Assume that we found $F_n \in \Dert_2$ such that 
$\Phi_n=(F_n^{12,3})^{-1}(F_n^{1,2})^{-1}F_n^{2,3}F_n^{1,23}$
is equal to $\Phi$ modulo terms of degree greater than $n$.
Then,  
$F_n^{2,3}F_n^{1,23}(x+y+z)=F_n^{1,2}F_n^{12,3}(x+y+z)$ modulo terms of degree
greater than $n+1$, and $F_n(x,y)=\ch_s(x,y)$ modulo terms of degree greater than
$n+1$. Since $F_n^{123,4}\Phi_n^{1,2,3}=\Phi_n^{1,2,3}F_n^{123,4}$ and 
$F_n^{1,234}\Phi_n^{2,3,4}= \Phi_n^{2,3,4}F_n^{1,234}$ modulo terms of 
degree greater than $n+1$, $\Phi_n$ satisfies the pentagon equation
modulo terms of degree greater than $n+1$.
Write $\Phi_n=\exp(\sum_{k=2}^\infty \psi_k)$, where $\psi_k=\phi_k$ for $k\leq n$ and
denote $\varphi=\phi_{n+1}-\psi_{n+1}$. The pentagon equation for $\Phi$
and the pentagon equation modulo terms of degree greater then $n+1$ for $\Phi_n$
imply $\d \varphi=0$. Hence, by Theorem \ref{th:hd}, $\varphi=\d u$ for 
a unique element $u \in \dert_2$ of degree $n+1$. Put $F_{n+1}=F_n \exp(u)$. It satisfies equation
$\Phi=(F_{n+1}^{12,3})^{-1}(F_{n+1}^{1,2})^{-1}F_{n+1}^{2,3}F_{n+1}^{1,23}$
modulo terms of degree greater than $n+1$. By induction, we construct a unique
$F$ which solves equation $\Phi=(F^{12,3})^{-1}(F^{1,2})^{-1}F^{2,3}F^{1,23}$
and has $f_1=sr/2$, as required.
By Proposition \ref{prop:kv=>pent}, the element $F$ solves the KV problem,
$F \in \Sol_s(\HKV)$. 
\end{proof}

Theorem \ref{th:pent=>kv} implies that the Kashiwara-Vergne problem has solutions 
if an only if the pentagon equation has solutions $\Phi \in \KV_3$ with 
$\pi(\phi_2)=1/8$. The next proposition provides a tool extracting the Duflo
function of an element $F \in \Sol(\HKV)$ from the corresponding solution
of the pentagon equation.

\begin{proposition}  \label{prop:extra2}
Let $\Phi=\exp(\phi) \in \KV_3$ be a solution of the pentagon equation with
$\pi(\phi_2)=1/8$, and let $F \in \Sol(\HKV)$ be a solution of 
equation \eqref{eq:phiffff}. Denote $\phi=(A,B,C)$, and
$B(x,0,z)_{x-lin}=h(\ad_z) x$ for $h\in x \k[[x]]$. Then, the Duflo
function of $F$ satisfies equation $f'(x)=h(x)$. 
\end{proposition}

\begin{proof}
Let $F=\exp(u)$ with $u=(a,b)$. Put  $a(x,y)= a_0 y + \alpha(\ad_y) x + \dots$ and
$b(x,y)=b_0 y + \beta(\ad_y) x + \dots$. Then, by Proposition \ref{prop:extra1},
the Duflo function associated to $F$ is a solution of equation $f'=\beta-\alpha$.

Denote 
$$
\begin{array}{lll}
u^l=u^{1,2}+u^{12,3} & = & (a(x,y)+a(x+y,z),b(x,y)+a(x+y,z), b(x+y,z)) \\
u^r=u^{2,3}+u^{1,23} & = & (a(x,y+z), a(y,z)+b(x,y+z), b(y,z)+b(x,y+z)) ,
\end{array}
$$
and observe that $\phi=\ch(-u^l, u^r)$. The contribution of 
$u^r-u^l$ in $B(x,0,z)_{x-lin}$ is equal to $\beta(\ad_z)x-\alpha(\ad_z)x$.
Note that the linear in $z$ contributions
in both $u^l$ and $u^r$ are of the form $(z,z,0)$. Since
$$
[(z,z,0), (0, h(\ad_z)x, 0)]=(0, h(\ad_z)[x,z]+[z, h(\ad_z)x],0)=0,
$$
we conclude that the nonlinear terms in the Campbell-Hausdorff
series $\ch(-u^l,u^r)$ do not contribute in $B(x,0,z)_{x-lin}$,
and $h(x)=\beta(x)-\alpha(x)$. Hence, $f'(x)=h(x)$, as required.
\end{proof}

\section{$\Z_2$-symmetry of the KV problem and hexagon equations}  \label{sec:sym}

In this Section we introduce an involution on $\tau$ the set of 
solutions of the generalized KV problem, and show that the corresponding
solutions of the pentagon equation verify a pair of hexagon equations.

\subsection{The automorphism $R$ and the Yang-Baxter equation} 
Let $R \in \Dert_2$ be an automorphism of $\lie_2$
defined on generators by $R(x)=e^{-\ad_y} x, R(y)=y$. Note that $R=\exp(r)$ for
$r=(y,0) \in \dert_2$, and
$$
R\, (\ch(y,x))=\ch(y, \exp(-\ad_y)x)=\ch(x,y).
$$
Denote by $\theta$ the inner derivation of $\lie_2$ with generator
$\ch(x,y)$. That is, for $a \in \lie_2$ we have $\theta(a)=[a, \ch(x,y)]$.
Note that the derivation $t=(y,x) \in \dert_2$ is an inner derivation
of $\lie_2$ with generator $x+y$. Indeed, $t(x)=[x,y]=[x,x+y]$ and
$t(y)=[y,x]=[y,x+y]$. Let $F\in \Dert_2$ be a solution of the first KV equation,
$F(x+y)=\ch(x,y)$. Then, $F tF^{-1}=\theta$. Indeed, for $a \in \lie_2$ we have
$$
FtF^{-1} (a)=F([F^{-1}(a), x+y])=[a, F(x+y)]=[a,\ch(x,y)]=\theta(a).
$$ 

\begin{proposition}
$R R^{2,1}=\exp(\theta)$.
\end{proposition}

\begin{proof}
Note that $R^{2,1}(x)=x$ and $R^{2,1}(y)=e^{-\ad_x} y$.
We compute,
$$
R R^{2,1} (x)=R(x)=\exp(-\ad_y) x=\exp(-\ad(\ch(x,y))) x,
$$
and
$$
R R^{2,1} (y)=R(\exp(-\ad_x)y)=\exp(-\ad(\exp(-\ad_y)x)) y=\exp(-\ch(x,y)) y,
$$
as required.
\end{proof}

\begin{proposition} \label{prop:YB}
The element $R$ satisfies the Yang-Baxter equation,
$$
R^{1,2}R^{1,3}R^{2,3}=R^{2,3}R^{1,3}R^{1,2}.
$$
\end{proposition}

\begin{proof}
In components, we have $R^{1,2}=(\exp(-\ad_y),1,1), R^{1,3}=(\exp(-\ad_z),1,1)$
and $R^{2,3}=(1,\exp(-\ad_z),1)$. One easily computes both the left
hand side and the right hand side of the Yang-Baxter equation on
generators $y$ and $z$,
$ z\mapsto z$ and $y  \mapsto \exp(-\ad_z)y$. We compute the action of the 
left hand side on $x$:
$$
R^{1,2}R^{1,3}R^{2,3}(x)=R^{1,2}R^{1,3}(x)=R^{1,2}(\exp(-\ad_z) x)=
\exp(-\ad_z)\exp (-\ad_y) x,
$$
and the action of the right hand side,
$$
\begin{array}{lll}
R^{2,3}R^{1,3}R^{1,2}(x) & = & R^{2,3}R^{1,3}(\exp(-\ad_y)x) \\
 & = & R^{2,3}(\exp(-\ad_y) \exp(-\ad_z) x) \\
 & = & \exp(-\ad_z)\exp (-\ad_y) x 
\end{array}
$$
which completes the proof.
\end{proof}

\begin{proposition} \label{prop:R=RR}
$R^{12,3}=R^{1,3}R^{2,3}$.
Let $F \in \Dert_2$ be a solution of equation $F(x+y)=\ch(x,y)$. Then,
$F^{2,3}R^{1,23}(F^{2,3})^{-1}=R^{1,2}R^{1,3}$.
\end{proposition}

\begin{proof}
For the first equation, note that both sides
are represented by the automorphism $(\exp(-\ad_z), \exp(-\ad_z), 1) \in \Dert_3$.

For the second equation, both the left hand side and the right
hand side preserve generators $y$ and $z$, $y \mapsto y, z \mapsto z$. 
It remains to compute the action on $x$:
$$
F^{2,3}R^{1,23}(F^{2,3})^{-1}(x)=F^{2,3}R^{1,23}(x)=
F^{2,3}(\exp(-\ad_{y+z}) x)=\exp(-\ch(y,z))x,
$$
and the same for the right hand side
$$
R^{1,2}R^{1,3}(x)=R^{1,2}(\exp(-\ad_z)x)=\exp(-\ad_z)\exp(-\ad_y)x=
\exp(-\ch(y,z))x,
$$
as required.
\end{proof}

\subsection{Involution on $\Sol(\HKV)$}
In this Section we introduce and study a certain involution
on the set of solutions of the KV problem.

\begin{proposition}
Let $F \in \Sol(\HKV)$. Then, $\tau(F)=R F^{2,1} e^{-t/2}$ is a solution of the 
KV problem, $\tau(F) \in \Sol(\HKV)$. The map $\tau$ is an involution, $\tau^2=1$.
\end{proposition}

\begin{proof}
We compute, 
$$
\tau(F) (x+y)= R F^{2,1} e^{-t/2} (x+y) = R F^{2,1} (x+y)=
R (\ch(y,x))=\ch(x,y).
$$
Furthermore,
$$
j(\tau(F))=j(R F^{2,1} e^{-t/2})=R \cdot j(F^{2,1}).
$$
Here we used that $\dv(r)=\dv(t)=0$ and $j(R)=j(\exp(-t/2))=0$.
Let $f \in x^2\k[[x]]$ such that $j(F)=\Tr(f(x)-f(\ch(x,y))+f(y))$.
Then, $j(F^{2,1})=\Tr(f(x)-f(\ch(y,x))+f(y))$ and
$R \cdot j(F^{2,1})=\Tr(f(x)-f(\ch(x,y))+f(y))=j(F)$.
Hence, $\tau(F)$ is a solution of the KV problem.

Finally,
$$
\tau^2(F)=R \tau(F)^{2,1} e^{-t/2}=
R R^{2,1} F e^{-t} =e^\theta F e^{-t}=F,
$$
where we used $t^{2,1}=t$, $RR^{2,1}=\exp(\theta)$ and $FtF^{-1}=\theta$.
We conclude that $\tau^2=1$, and $\tau$ defines an involution on
$\Sol(\HKV)$.
\end{proof}

\begin{proposition}   \label{prop:tauphi}
Let $F\in \Sol(\HKV)$ and let $\Phi_F$ be the corresponding solution
of the pentagon equation. Then,
$$
\Phi_{\tau(F)} = (\Phi_F^{3,2,1})^{-1} .
$$ 
\end{proposition}

\begin{proof}
We compute,
$$
\begin{array}{lll}
\Phi_{\tau(F)} & = & e^{t^{12,3}/2} (F^{3,21})^{-1}(R^{12,3})^{-1}
e^{t^{1,2}/2} (F^{2,1})^{-1} (R^{1,2})^{-1} R^{2,3} F^{3,2}
e^{-t^{2,3}/2} R^{1,23} F^{32,1} e^{-t^{1,23}/2} \\
& = & e^{c/2} (F^{3,21})^{-1}(R^{12,3})^{-1}
(F^{2,1})^{-1} (R^{1,2})^{-1} R^{2,3} F^{3,2} R^{1,23} F^{32,1} e^{-c/2} \\
& = & e^{c/2} (F^{3,21})^{-1} (F^{2,1})^{-1} 
(R^{2,3})^{-1} (R^{1,3})^{-1} (R^{1,2})^{-1} R^{2,3} R^{1,3} R^{1,2}
F^{3,2} F^{32,1} e^{-c/2} \\
& = & e^{c/2} (F^{3,21})^{-1} (F^{2,1})^{-1} F^{3,2} F^{32,1} e^{-c/2} = 
e^{c/2} (\Phi^{3,2,1})^{-1} e^{-c/2} = (\Phi^{3,2,1})^{-1} .
\end{array}
$$
Here in passing from the first to the second line we used that
$g^{1,2}h^{12,3}=h^{12,3}g^{1,2}$ for $g \in \Sder_2, h\in \Dert_2$,
and the definition of the element $c=t^{1,2}+t^{1,3}+t^{2,3} \in \t_3$;
Proposition \ref{prop:R=RR} in the passage from the second to the third line;
and finally the Yang-Baxter equation (Proposition \ref{prop:YB}) and the fact that $c$ is
central in $\kv_3$ in the passage from the third to the fourth line.
\end{proof}

\begin{proposition}  \label{prop:kappatauF}
Let $F \in \Sol(\HKV)$ and $\kappa(F)=(A(x,y),B(x,y)) \in \dert_2$.
Then,
\begin{equation}  \label{eq:kappatauF}
\kappa(\tau(F))=\left( e^{\ad_x} B(y,x) +\frac{1}{2}(\ch(x,y)-x),
e^{-\ad_y} A(y,x) - \frac{1}{2}(\ch(x,y)-y) \right) .
\end{equation}
\end{proposition}

\begin{proof}
We compute,
$$
\kappa(\tau(F))=\frac{d \tau(F)_s}{ds}|_{s=1} \, \tau(F)^{-1}=
r + R \, \frac{d F^{2,1}_s}{ds}|_{s=1} (F^{2,1})^{-1} R^{-1}
-\frac{1}{2} RF^{2,1} t (F^{2,1})^{-1}R^{-1} ,
$$
where we used that $dR_s R_s^{-1}=r=(y,0) \in \dert_2$. In the last term,
$F^{2,1} t (F^{2,1})^{-1}$ is the inner derivation with generator $\ch(y,x)$,
and $RF^{2,1} t (F^{2,1})^{-1}R^{-1}$ is an inner derivation with
generator $\ch(x,y)$. With our normalization condition, it is represented
by $(\ch(x,y)-x, \ch(x,y)-y) \in \dert_2$. 

Finally, for the middle term $R \kappa(F)^{2,1} R^{-1}$ we compute,
$$
\begin{array}{lll}
R(A,B)^{2,1}R^{-1}(x) & = & R (B(y,x), A(y,x)) e^{\ad_y}(x) \\
& = & R(e^{\ad_y}[x, B(y,x)] + e^{\ad_y} [A(y,x),x] - [A(y,x), e^{\ad_y}(x)]) \\
& = & [x, B(y,x) + (e^{-\ad_y} - 1) A(y,x)] \\
& = & [x, e^{\ad_x} B(y,x) + \ch(x,y) -x-y] .
\end{array}
$$
Here in the passage to the last line we have used equation \eqref{eq:oldkv1}
(with $x$ and $y$ exchanged).
For the action on $y$ we compute, 
$$
R(A,B)^{2,1}R^{-1}(y)=R(B(y,x)A(y,x))(y)=R([y,A(y,x)])=[y, e^{-\ad_y} A(y,x)] .
$$
By adding up all three terms we obtain,
$$
\begin{array}{lll}
\kappa(\tau(F)) & = & (e^{\ad_x} B(y,x) + \ch(x,y) -x-y, e^{-\ad_y} A(y,x)) \\
 & + & (y,0) - \frac{1}{2} (\ch(x,y)-x, \ch(x,y)-y) \\
& = & ( e^{\ad_x} B(y,x) +\frac{1}{2}(\ch(x,y)-x),
e^{-\ad_y} A(y,x) - \frac{1}{2}(\ch(x,y)-y) ,
\end{array}
$$
as required.
\end{proof}

\begin{remark}
Symmetry \eqref{eq:kappatauF} has been introduced in \cite{kv} (see discussion
after Proposition 5.3).
\end{remark}

\subsection{Symmetric solutions of the KV problem}

\begin{definition}
An element $F\in \Sol(\HKV)$ is called {\em a symmetric solution} of the 
generalized Kashiwara-Vergne conjecture if $\tau(F)=F$.
\end{definition}

We shall denote the set of symmetric solutions by $\Sol^\tau(\HKV)$.
Since the map $\kappa: \Dert_2 \to \dert_2$ is a bijection,
$\tau(F)=F$ if and only if $\kappa(\tau(F))=\kappa(F)$.
That is, $\kappa(F)=(A(x,y), B(x,y))$ satisfies the (equivalent)
linear equations
$$
A(x,y)=e^{\ad_x} B(y,x) + \frac{1}{2}(\ch(x,y)-x) \, , \,
B(x,y)=e^{-\ad_y} A(y,x) - \frac{1}{2}(\ch(x,y)-y) .
$$
Since equations \eqref{eq:oldkv1} and \eqref{eq:oldkv2} are linear
in $A$ and $B$, one can average an arbitrary solution to obtain
a symmetric solution $\tilde{F}$ with $\kappa(\tilde{F})=(\kappa(F)+\kappa(\tau(F)))/2$.

The involution $u\mapsto u^{2,1}$ acts on the Lie algebra $\hkv_2$,
and it lifts to the group $\HKV_2$. We shall denote the corresponding 
invariant subalgebra by $\hkv_2^{sym} \subset \hkv_2$ and the invariant subgroup
by $\HKV_2^{sym} \subset \HKV_2$.

\begin{proposition}
The group $\HKV_2^{sym}$ acts on the set $\Sol^\tau(\HKV)$ by multiplications
on the right. This action is free and transitive.
\end{proposition}

\begin{proof}
Let $g\in \HKV_2^{sym}$ and $F\in \Sol^{\tau}(\HKV)$. By Theorem \ref{th:solhkv},
$Fg \in \Sol(\HKV)$. By applying $\tau$ we obtain
$$
\tau(Fg)=R F^{2,1} g^{2,1} e^{-t/2}= R F^{2,1} e^{-t/2} g=\tau(F) g=Fg.
$$
Hence, $Fg \in \Sol^{\tau}(\HKV)$.

Consider two elements $F_1, F_2 \in \Sol^{\tau}(\HKV)$. We denote $g=F_1^{-1} F_2$
and compute
$$
g^{2,1}=(F_1^{-1} F_2)^{2,1}=(R^{-1} F_1 e^{t/2})^{-1} (R^{-1} F_2 e^{t/2})=
e^{-t/2} (F_1^{-1} F_2) e^{t/2}=e^{-t/2} g e^{t/2}=g,
$$
as required.
\end{proof}

\begin{remark}
Note that the element $t=(y,x)$ as well as the image of the 
injection $\nu: \grt \to \hkv_2$ is contained in $\hkv_2^{sym}$.
In fact, it is not known whether any non-symmetric 
elements of $\hkv_2$ exist. If correct, Conjecture stated in the end of 
Section~\ref{sec:KVlie} would imply $\hkv_2 = \hkv_2^{sym}$. 
\end{remark}

\begin{proposition}
Let $F \in \Sol^\tau(\HKV)$, and let $\Phi \in \KV_3$ be the corresponding
solution of the pentagon equation. Then,
\begin{equation} \label{eq:inverse}
\Phi^{1,2,3} \Phi^{3,2,1} = e,
\end{equation}
\begin{equation}  \label{eq:hex1}
e^{(t^{1,3}+t^{2,3})/2}=\Phi^{2,1,3} e^{t^{1,3}/2} (\Phi^{2,3,1})^{-1} 
e^{t^{2,3}/2} \Phi^{3,2,1}
\end{equation}
and
\begin{equation}  \label{eq:hex2}
e^{(t^{1,2}+t^{1,3})/2}=(\Phi^{1,3,2})^{-1} e^{t^{1,3}/2} \Phi^{3,1,2} 
e^{t^{1,2}/2} (\Phi^{3,2,1})^{-1}
\end{equation}
\end{proposition}

\begin{proof}
Equation \eqref{eq:inverse} follows by Proposition \ref{prop:tauphi}.
In order to prove equation \eqref{eq:hex1} recall that 
$R^{12,3}=R^{1,3}R^{2,3}=(\exp(-\ad_z), \exp(-\ad_z), 1) \in \Dert_3$. Furthermore,
this automorphism commutes with $g^{1,2}$ for any $g\in \Dert_2$.
In particular, we have $F^{2,1} R^{12,3} (F^{2,1})^{-1}= R^{1,3} R^{2,3}$.
By substituting $R=Fe^{t/2}(F^{2,1})^{-1}$ we obtain,
$$
F^{2,1} R^{12,3} (F^{2,1})^{-1} = F^{2,1} F^{21,3} e^{(t^{1,2}+t^{1,3})/2}
(F^{3,12})^{-1} (F^{2,1})^{-1},
$$
and
$$
\begin{array}{lll}
R^{1,3} R^{2,3} & = & F^{1,3} e^{t^{1,3}/2}(F^{3,1})^{-1} F^{2,3}e^{t^{2,3}/2}(F^{3,2})^{-1} \\
& = & F^{1,3} F^{2,13} e^{t^{1,3}/2}(F^{2,31})^{-1} (F^{3,1})^{-1} 
F^{2,3} F^{23,1} e^{t^{2,3}/2}(F^{32,1})^{-1} (F^{3,2})^{-1} .
\end{array}
$$
A comparison of these two equations yields equation \eqref{eq:hex1}.
Equation \eqref{eq:hex2} follows by applying the $(13)$-permutation to equation
\eqref{eq:hex1} and by using the inversion formula \eqref{eq:inverse}.
\end{proof}

\begin{remark}
Equations \eqref{eq:hex1} and \eqref{eq:hex2} are called as {\em hexagon equations}.
They were first introduced in \cite{dr} (see equations (2.14a) and (2.14b)).
\end{remark}

\section{Associators}  \label{sec:associators}

In this Section we consider joint solutions of pentagon and hexagon 
equations called {\em associators} (with values in the group $\KV_3$). 
We show that Drinfeld's associators defined in \cite{dr} make part of this set,
and we use this fact to give a new proof of the KV conjecture.

\subsection{Associators with values in $KV_3$ and Drinfeld's associators}

\begin{definition}
An element $\Phi \in \KV_3$ is an associator if it satisfies 
the pentagon equation \eqref{eq:pent}, hexagon equations
\eqref{eq:hex1} and \eqref{eq:hex2} and the inversion property
\eqref{eq:inverse}.
\end{definition}

\begin{proposition}
Let $\Phi=\exp(\phi) \in \KV_3$ be an associator.
Then, $\pi(\phi_2)=1/8$.
\end{proposition}

\begin{proof}
The degree two component of the hexagon equation \eqref{eq:hex1} reads
$$
\frac{1}{8} \, [t^{1,3}, t^{2,3}] + \phi_2^{2,1,3} - \phi_2^{2,3,1}
+\phi_2^{3,2,1} =0.
$$
Note that $[t^{1,3}, t^{2,3}]=([y,z], [z,x], [x,y])$ which implies
$\pi([t^{1,3}, t^{2,3}])=3$. Also observe that  $\pi(\phi_2^{2,3,1})=\pi(\phi_2)$
and $\pi(\phi_2^{2,1,3})=\pi(\phi_2^{3,2,1})=-\pi(\phi_2)$. We conclude
that $3\pi(\phi_2)=3/8$ and $\pi(\phi_2)=1/8$, as required.
\end{proof}

\begin{proposition}  \label{prop:Fsym}
Let $\Phi=\exp(\phi) \in \KV_3$ be a solution of equations 
\eqref{eq:pent} and \eqref{eq:inverse} with $\pi(\phi_2)=1/8$.
Then, each $F\in \Sol(\HKV)$ which verifies equation \eqref{eq:phiffff}
is a symmetric solution of the KV problem, $F \in \Sol^\tau(\HKV)$.
\end{proposition}

\begin{proof} 
Theorem \ref{th:pent=>kv} implies that equation \eqref{eq:phiffff} admits
solutions $F \in \Sol(\HKV)$.
By Proposition \ref{prop:tauphi}, $\Phi_{\tau(F)}=(\Phi_F^{3,2,1})^{-1}=\Phi_F$.
Hence, by Proposition \ref{prop:F1F2}, $\tau(F)=F\exp(\lambda t)$ for some $\lambda \in \k$.
The degree one component of this equation reads
$r+f_1^{2,1}-t/2=f_1+\lambda t$. Since $f_1=r/2 + \alpha t$ for some
$\alpha \in \k$, we have $r+f_1^{2,1}-f_1=t/2$ and $\lambda=0$.
In conclusion, $\tau(F)=F$, as required.
\end{proof}

Recall that by Proposition \ref{prop:tn} Lie algebras $\t_n$ inject into
$\kv_n$. In particular, $\t_3$ injects into $\kv_3$, and the corresponding
group $T_3$ is a subgroup of $\KV_3$.

\begin{definition}
An associator $\Phi \in \KV_3$ is called a Drinfeld's associator if
$\Phi \in T_3$. 
\end{definition}

Drinfeld's associators can be defined without referring to
the Lie algebras $\dert_n$ and $\kv_n$ since both simplicial and 
coproduct maps restrict to Lie subalgebras $\t_n$ in a natural way.
In \cite{dr'} Drinfeld proved the following theorem:

\begin{theorem}
The set of Drinfeld's associators is non empty.
\end{theorem}

This implies the following result:

\begin{theorem}  \label{th:kvneq0}
The set of symmetric solutions of the KV problem $\Sol^\tau(\HKV)$
is non empty.
\end{theorem}

\begin{proof}
Each Drinfeld's associator $\Phi=\exp(\phi)$ is an associator with values
in $\KV_3$ with $\pi(\phi_2)=1/8$. Then, by Theorem \ref{th:pent=>kv}, there is an element
$F=\exp(f) \in \Dert_2$ with $f_1=r/2$ which solves equation \eqref{eq:phiffff}.
By Proposition \ref{prop:kv=>pent} this automorphism is a solution of the KV problem,
and by Proposition \ref{prop:Fsym} this solution is symmetric.
\end{proof}

\begin{remark}
The KV problem has been settled in \cite{am}. 
The solution is based on the Kontsevich deformation quantization scheme \cite{kont},
and on the earlier work of the second author \cite{torossian}. Theorem \ref{th:kvneq0}
gives a new proof of the KV conjecture by reducing it to
the existence theorem  for Drinfeld's associators.
\end{remark}

\begin{proposition} \label{prop:extra3}
Let $\Phi=\exp(\phi) \in T_3$ be a Drinfeld's associator, and
let $F \in \Sol(\HKV)$ be a solution of the KV problem which
satisfies equation \eqref{eq:phiffff}.  Write
$\phi=h(\ad_{t^{2,3}}) t^{1,2} + \dots$, where $h\in x\k[[x]]$, and
$\dots$ stand for terms which contain at least two generators $t^{1,2}$. 
Then, the Duflo function associated to $F$ satisfies equation $f'(x)=h(x)$.
\end{proposition}

\begin{proof}
By putting $y=0$ we obtain $t^{1,2}=(y,x,0) \mapsto (0,x,0)$
and $t^{2,3} =(0,z,y) \mapsto (0,z,0)$. Hence,
$$
\phi(t^{1,2}, t^{2,3})_{y=0}=(0, \phi(x,z), 0).
$$
In particular, for $\phi=(A,B,C)$, we have $B(x,0,z)_{x-lin}=h(\ad_z)x$.
Then, by Proposition \ref{prop:extra2}, we obtain $f'(x)=h(x)$, as
required.
\end{proof}

\begin{example}
Consider the Knizhnik-Zamolodchikov associator (with values in $T_3$)
constructed in Drinfeld. Equation (2.15) of \cite{dr}
yields the function $h(x)$: 
$$
h(x)= - \sum_{n=2}^\infty \, \frac{\zeta(n)}{(2\pi i)^n} \, x^{n-1} .
$$
Note that our associators are obtained by taking an inverse 
of associators the in Drinfeld's paper. The Duflo function
corresponding to the Knizhnik-Zamolodchikov associator is given by
$$
f(x)= - \sum_{n=2}^\infty \, \frac{\zeta(n)}{n (2\pi i)^n} \, x^n =
\frac{\gamma}{2\pi i} x - \ln\left(\Gamma\left(1-\frac{x}{2\pi i}\right)\right) .
$$
Here $\gamma$ is the Euler's constant, and the term
$\gamma x/2\pi i$  cancels the linear part  in the logarithm of the 
$\Gamma$-function. Formula for $f(x)$ matches (up to a sign change) 
the expression $\ln(F_{nice}(x))$ in \cite{kont'}.
\end{example}

\subsection{Actions of the group $\GRT$}
Let $\LLie_n$ be a group associated to the Lie algebra $\lie_n$
(such that $a\cdot b=\ch(a,b)$).
Then, one can view the Grothendieck-Teichm\"uller group $\GRT$
as a subset of $\LLie_2$ defined by a number of relations
(see Section 5 of \cite{dr}), and equipped with the new
multiplication,
$$
(h_1 *_{\GRT} \, h_2)(x,y)=h_1(x, h_2(x,y)yh_2^{-1}(x,y)) h_2(x,y).
$$ 

\begin{remark}
Note that we have chosen to act on the second
argument of the function $h$ rather than on the first one (as in \cite{dr}).
\end{remark}

Let $\psi \in \grt$ and consider a one parameter subgroup of $\GRT$ defined
by $\psi$, $h_s=\exp_{\GRT}(s\psi)$. Write $h_t=h_{t-s} *_{\GRT} h_s$
and differentiate in $t$ at $t=s$ to obtain
$$
\frac{dh_s(x,y)}{ds} = \psi(x, h_s(x,y) y h_s(x,y)^{-1}) h_s(x,y).
$$
This differential equation together with the initial condition
$h_0(x,y)=1$ defines the exponential function $\exp_{\GRT}$
in a unique way.

\begin{proposition}  \label{prop:new}
Let $\psi \in \grt$, $h=\exp_{\GRT}(\psi) \in \GRT$ and 
$g=\exp(\nu(\psi)) \in \HKV_2$. Then,
$$
\hat{g}=(g^{12,3})^{-1} (g^{1,2})^{-1} g^{2,3} g^{1,23} = h(t^{1,2},t^{2,3}) \in \KV_3 .
$$
\end{proposition}

\begin{proof}
First, observe that for $g \in \Sder_2$, $g^{1,2}$ commutes with $g^{12,3}$, and $g^{2,3}$ 
commutes with $g^{1,23}$. Hence, the maps $g \mapsto g^l=g^{1,2}g^{12,3}$ and 
$g \mapsto g^r=g^{2,3}g^{1,23}$ are group homomorphisms mapping $\Sder_2$ to $\Sder_3$.

Next, replace $\psi$ by $s\psi$ and consider the derivative in $s$ 
of $\hat{g}_s=(g^l_s)^{-1}g^r_s$:
$$
\begin{array}{lll}
\frac{d\hat{g}_s}{ds} & = & (g^l_s)^{-1} \left( \frac{dg^r_s}{ds} \, (g^r_s)^{-1} - 
\frac{dg^l_s}{ds} \, (g^l_s)^{-1} \right) \, g^r_s \\
& = & (g^l_s)^{-1} (\d \nu(\psi)) g^r_s \\
& = & (g^l_s)^{-1} \psi(t^{1,2}, t^{2,3}) g^r_s \\
& = & \psi(t^{1,2}, (g^l_s)^{-1} t^{2,3} g^l_s) (g^l_s)^{-1} g^r_s \\
& = & \psi(t^{1,2}, (g^l_s)^{-1}g^r_s t^{2,3} (g^r_s)^{-1} g^l_s) \hat{g}_s \\
& = & \psi(t^{1,2}, \hat{g}_s t^{2,3} (\hat{g}_s)^{-1}) \hat{g}_s .
\end{array}
$$
Obviously, $\hat{g}_0=e \in \KV_3$. We conclude that $h(t^{1,2}, t^{2,3})$ and
$\hat{g}$ satisfy the same first order linear ordinary differential equation with
the same initial condition. Hence, they coincide, as required.
\end{proof}

The Lie algebra homomorphism $\nu: \grt \rightarrow \hkv_2$ gives rise to a subgroup
of $\HKV_2$ isomorphic to $\GRT$. The group $\HKV_2$ acts on the set of solutions
of the KV problem, and on the set of associators with values in $\KV_3$ 
(see equation \eqref{eq:twist}). In \cite{dr} (see Section 5)
Drinfeld defines a free and transitive action of the group $\GRT$ on the
set of associators with values in $T_3$. This action is given by the
following formula,
\begin{equation}  \label{eq:grtaction}
g: \Phi(t^{1,2},t^{2,3}) \mapsto \Phi(t^{1,2}, gt^{2,3}g^{-1}) g,
\end{equation}
where $g=\exp_{\GRT}(\psi) \in \GRT$ and $\Phi \in T_3$ are viewed as elements 
of the group  $\LLie_2(t^{1,2}, t^{2,3})$. 
The following proposition relates these two actions.

\begin{proposition}
When restricted to the set of Drinfeld's associators, the 
action of the group $\GRT$ on associators with values in 
$\KV_3$ coincides with the canonical action \eqref{eq:grtaction}.
\end{proposition}

\begin{proof}
Let $g \in \HKV_2$ and rewrite the action \eqref{eq:twist} on
$\Phi(t^{1,2},t^{2,3}) \in T_3$ as follows,
$$
\Phi \cdot g = (g^{12,3})^{-1}(g^{1,2})^{-1} \Phi(t^{1,2},t^{2,3}) g^{2,3}g^{1,23}
=\Phi(t^{1,2}, \hat{g} t^{2,3} \hat{g}^{-1}) \hat{g} ,
$$
for $\hat{g}=(g^{12,3})^{-1}(g^{1,2})^{-1} g^{2,3}g^{1,23}$. 
Let $\psi \in \grt$ and $g=\exp(\nu(\psi))$. Then, by Proposition \ref{prop:new} 
we have $\hat{g}=(\exp_{\GRT}(\psi))(t^{1,2},t^{2,3})$, and the action
\eqref{eq:twist} coincides with the canonical action
\eqref{eq:grtaction}.
\end{proof}

\begin{remark}
If Conjecture of Section \ref{sec:KVlie} is correct, we have 
$\HKV_2 \cong \k t \times \nu(\GRT)$, where the additive group 
$\k t$ injects into $\HKV_2$ via the exponential map, 
$\lambda t \mapsto \exp(\lambda t)$.
In particular, this implies $\HKV_2=\HKV_2^{sym}$ since 
both $\k t$ and $\nu(\GRT)$ are contained in $\HKV_2^{sym}$.
Note that the action of $\k t$ on associators is trivial, and 
the action of $\GRT$ on the set of Drinfeld's
associators is transitive. The action of $\HKV^{sym}_2$ on associators with
values $\KV_3$ is also transitive, and we conclude that all associators
with values in $\KV_3$ are Drinfeld's associators.
\end{remark}

\begin{remark}
For Drinfeld's associators,
Furusho \cite{furusho} showed that the hexagon equations \eqref{eq:hex1}, \eqref{eq:hex2}
and the inversion property \eqref{eq:inverse} follow from the pentagon equation and the
normalization condition $\pi(\phi_2)=1/8$. 
In the case of associators with values in $\KV_3$,
Proposition \ref{prop:Fsym} shows that the hexagon equations \eqref{eq:hex1}, 
\eqref{eq:hex2} follow from the pentagon equation, the inversion property and the
normalization condition $\pi(\phi_2)=1/8$. If we assumed  $\HKV_2=\HKV_2^{sym}$, 
the inversion property would be automatic, and we would get the analogue of
Furusho's result for associators with values in $\KV_3$. 
If Conjecture of Section \ref{sec:KVlie} holds true, we recover
the Furusho's result.
\end{remark}

\vskip 0.3cm

{\sc Appendix: proof of Proposition \ref{prop:dpsi}}

\vskip 0.2cm

In this Appendix we give a proof of Proposition \ref{prop:dpsi}.
It is inspired by the proof of Proposition 5.7 in \cite{dr}.
 
Denote $\d \Psi=(a,b,c)$. We have,
$$
\begin{array}{lll}
a & = & -\psi(-x-y,x)+\psi(-x-y-z,x)-\psi(-x-y-z,x+y), \\
b & = & -\psi(-x-y,y)+\psi(-x-y-z,y+z)-\psi(-x-y-z,x+y)+\psi(-y-z,y), \\
c & = & \psi(-x-y-z,y+z) -\psi(-x-y-z,z) +\psi(-y-z,z).
\end{array}
$$
Let $\g$ be the semi-direct sum of $\dert_3$ and $\lie_3$. The following formulas
define an injective Lie algebra homomorphism of $\t_4$ to $\g$:
$$
\begin{array}{lll}
t^{1,2} \mapsto (y,x,0) \in \dert_3, & t^{1,3} \mapsto (z,0,x) \in \dert_3, &
t^{2,3} \mapsto (0,z,y) \in \dert_3, \\
t^{1,4} \mapsto x \in \lie_3, & t^{2,4} \mapsto y \in \lie_3, &
t^{3,4} \mapsto z \in \lie_3.
\end{array}
$$
Indeed, $t^{1,2}, t^{1,3}$ and $t^{2,3}$ span a Lie subalgebra of $\dert_3$
isomorphic to $\t_3$, and $x,y$ and $z$ span an ideal of $\t_4$ isomorphic to a 
free Lie algebra with three generators.
It remains to check the Lie brackets between generators of these two 
Lie subalgebras. For instance, we compute,
$$
[t^{1,2}, t^{3,4}]=t^{1,2}(z)=0, \,\,\,
[t^{1,2}, t^{2,4}]=t^{1,2}(y)=[y,x]=[t^{2,4},t^{1,4}],
$$
as required.

Note that $(\d \Psi)(x)$ is the image of the following element of
$\t_4$,
$$
\begin{array}{ll}
 & [t^{1,4}, -\psi(-t^{1,4}-t^{2,4},t^{1,4})+\psi(-t^{1,4}-t^{2,4}-t^{3,4},t^{1,4}) \\
- & \psi(-t^{1,4}-t^{2,4}-t^{3,4},t^{1,4}+t^{2,4})] \\
= & [t^{1,4}, -\psi(t^{1,2}, t^{1,4})+\psi(t^{1,2}+t^{1,3}+t^{2,3}, t^{1,4}) -
\psi(t^{1,2}+t^{1,3}+t^{2,3}, t^{1,4} + t^{2,4})] \\
= & [t^{1,4}, -\psi(t^{1,2}, t^{1,4}) +\psi(t^{1,2}+t^{1,3}, t^{1,4}) -
\psi(t^{1,3}+t^{2,3}, t^{1,4} + t^{2,4})] \\
= & [t^{1,4}, -\psi(t^{2,3}, t^{1,2}+t^{2,4}) +\psi(t^{2,3}, t^{1,2})] \\ 
= & [t^{1,4}, \psi(t^{2,3}, t^{1,2})] = [\psi(t^{1,2}, t^{2,3}), t^{1,4}].
\end{array}
$$
Here in passing from the first to the second line we used the properties of 
central elements in $\t_3$ and $\t_4$. For instance, $t^{1,2}+t^{1,4}+t^{2,4}$ is central
in the Lie subalgebra (isomorphic to $\t_3$) spanned by $t^{1,2}, t^{1,4}$ and
$t^{2,4}$. In the passage from the second to the third line we used
the defining relations of the Lie algebra $\t_4$. For instance, in the second term we used
that $t^{2,3}$ has a vanishing bracket with $t^{1,4}$ and $t^{1,2}+t^{1,3}$.
In the passage from the second to the third line we used a $(3214)$ permutation
of the equation \eqref{eq:grt3}. Finally, in the last passage we again used the
defining relations of $\t_4$, and in particular the fact that $t^{1,4}$ has a vanishing
bracket with $t^{2,3}$ and with $t^{1,2}+t^{2,4}$. In conclusion, we have
$$
\d \Psi(x) =\psi(t^{1,2}, t^{2,3})(x).
$$
Similarly, $(\d \Psi)(y)$ is the image of the following element,
$$
\begin{array}{ll}
& [t^{2,4}, -\psi(-t^{1,4}-t^{2,4},t^{2,4})+\psi(-t^{1,4}-t^{2,4}-t^{3,4},t^{2,4}+ t^{3,4}) \\
- & \psi(-t^{1,4}-t^{2,4}-t^{3,4},t^{1,4}+t^{2,4})+\psi(-t^{2,4}-t^{3,4},t^{2,4})] \\
= & [t^{2,4}, -\psi(t^{1,2}, t^{2,4}) + \psi(t^{1,2}+t^{1,3}+t^{2,3}, t^{2,4}+ t^{3,4}) \\
- & \psi(t^{1,2}+t^{1,3}+t^{2,3}, t^{1,4}+ t^{2,4}) + \psi(t^{2,3}, t^{2,4})] \\
= & [t^{2,4}, -\psi(t^{1,3}, t^{1,2}+t^{1,4}) +\psi(t^{1,3}, t^{1,2})
+\psi(t^{1,3}, t^{2,3}+t^{3,4}) -\psi(t^{1,3}, t^{2,3})] \\
= & [t^{2,4}, -\psi(t^{1,3}, t^{1,2}+t^{1,4})+\psi(t^{1,3}, t^{2,3}+t^{3,4})-
\psi(t^{1,2}, t^{2,3})] \\
= &  [\psi(t^{1,2}, t^{2,3}), t^{2,4}].
\end{array}
$$
Here we used the $(1324)$ and $(3124)$ permutations of equation \eqref{eq:grt3}
as well as equation \eqref{eq:grt2} which implies $\psi(t^{1,2},t^{2,3})=
\psi(t^{1,2},t^{1,3})+\psi(t^{1,3},t^{2,3})$. Again, the conclusion is
$$
\d \Psi (y) = \psi(t^{1,2}, t^{2,3})(y).
$$
Finally, we represent $(\d \Psi)(z)$ as the image of the element
$$
\begin{array}{ll}
 & [t^{3,4}, \psi(-t^{1,4}-t^{2,4}-t^{3,4},t^{2,4}+t^{3,4}) -
\psi(-t^{1,4}-t^{2,4}-t^{3,4}, t^{3,4}) \\
+ & \psi(-t^{2,4}-t^{3,4}, t^{3,4})] \\
= & [t^{3,4}, \psi(t^{1,2}+t^{1,3}+t^{2,3},t^{2,4}+t^{3,4}) -
\psi(t^{1,2}+t^{1,3}+t^{2,3}, t^{3,4}) + \psi(t^{2,3}, t^{3,4})] \\
= & [t^{3,4}, \psi(t^{1,2}+t^{1,3}, t^{2,4}+t^{3,4}) -
\psi(t^{1,3}+t^{2,3}, t^{3,4}) + \psi(t^{2,3}, t^{3,4})] \\
= & [t^{3,4}, -\psi(t^{1,2}, t^{2,3}) + \psi(t^{1,2}, t^{2,3}+t^{2,4})] =
[\psi(t^{1,2}, t^{2,3}), t^{3,4}],
\end{array}
$$
where we used the equation \eqref{eq:grt3} (no permutation needed). We conclude
$$
\d \Psi(z) =\psi(t^{1,2}, t^{2,3})(z),
$$
and $\d \Psi = \psi(t^{1,2}, t^{2,3})$, as required.


\begin{thebibliography}{99}

\bibitem{aht} L. Albert, P. Harinck, C. Torossian,
Solution non universelle pour le probl\`eme $KV-78$,
preprint arXiv:0802.2049.

\bibitem{am} A. Alekseev, E. Meinrenken, On the Kashiwara-Vergne
conjecture, Invent. math. \textbf{164}, 615--634, (2006).

\bibitem{ap} A. Alekseev, E. Petracci, 
Low Order Terms of the Campbell-Hausdorff Series and 
the Kashiwara-Vergne Conjecture, J. Lie Theory \textbf{16}, no.3, 531--538, (2006);
Uniqueness in the Kashiwara-Vergne conjecture, preprint arXiv:math/0508077 
(Note that only the arXiv version of this paper contains statements which 
we are referring to).

\bibitem{ads} M. Andler, A. Dvorsky, S. Sahi, Deformation quantization 
and invariant distributions, 
C. R. Acad. Sci. Paris, S\'er. I, Math. \textbf{330}, 115--120, (2000).

\bibitem{ast} M. Andler, S. Sahi, C. Torossian, Convolutions of
invariant distributions: Proof of the Kashiwara-Vergne conjecture, 
Lett. Math. Phys. \textbf{69}, 177--203, (2004).

\bibitem{dr'} 
V.G. Drinfeld, Quasi-Hopf algebras, Leningrad Math. J., vol. \textbf{1}
no. 6, (1990).

\bibitem{dr}
V.G. Drinfeld, On quasi-triangular quasi-Hopf algebras and a group
closely connected with ${\rm Gal}(\overline{\mathbb{Q}}/\mathbb{Q})$,
Leningrad Math. J., vol. \textbf{2} no. 4, 829--860, (1991).

\bibitem{duflo} M. Duflo, Op\'erateurs diff\'erentiels bi-invariants sur
un groupe de Lie, Ann. Sci. \'Ec. Norm. Sur., IV. S\'er. \textbf{10},
265--288, (1977).

\bibitem{dynkin} E. B. Dynkin, 
Calculation of the coefficients in the Campbell-Hausdorff formula,
(Russian) Doklady Akad. Nauk SSSR (N.S.)  \textbf{57}, 323--326, (1947).

\bibitem{ENR}    M. Espie, J.-C. Novelli, G. Racinet, 
Formal computations about multiple zeta values.  From combinatorics to dynamical systems, 
1--16, IRMA Lect. Math. Theor. Phys., 3, de Gruyter, Berlin, (2003).

\bibitem{Ihara} Y. Ihara, The Galois representations arising
from $\mathbb{P}^1 - \{ 0,1,\infty\}$ and Tate twists of even
degree, Galois Groups over $\mathbb{Q}$, Math. Sci. Res. Inst. Publ.,
vol. \textbf{16}, Springer-Verlag, Berlin and New York, 199 --313, 
(1989). 

\bibitem{furusho}  H. Furusho,
Pentagon and hexagon equations, preprint arXiv:math/0702128.

\bibitem{kv} M. Kashiwara, M. Vergne, The Campbell-Hausdorff formula 
and invariant hyperfunctions, Invent. math. \textbf{47}, 249--272, (1978).

\bibitem{kont} M. Kontsevich, Deformation quantization of Poisson manifolds,
Lett. Math. Phys. \textbf{66}, 157--216, (2003).

\bibitem{kont'} M. Kontsevich,
Operads and Motives in Deformation Quantization,
Lett. Math. Phys. \textbf{48}, 35--72, (1999). 

\bibitem{mp} M. Podkopaeva, private communications

\bibitem{pt} M. Pevzner, C. Torossian, Isomorphisme de Duflo et la cohomologie 
tangentielle, J. Geom. Phys. \textbf{51}, 486--505, (2004).

\bibitem{racinet} G. Racinet, Doubles m\'elanges des polylogarithmes multiples aux
racines de l'unit\'e,  Publ. Math. Inst. Hautes \'Etudes Sci.  \textbf{95},
185--231, (2002).

\bibitem{rouviere} F. Rouvi\`ere, D\'emonstration de la conjecture de Kashiwara-Vergne
pour l'alg\`ebre ${\rm sl}(2)$, C.R. Acad. Sci. Paris, S\'er. I, Math. \textbf{292}.
657--660, (1981).

\bibitem{shoikhet} B. Shoikhet, Tsygan formality and Duflo formulas,
Math. Res. Lett. \textbf{10}, 763--775, (2003).

\bibitem{torossian} C. Torossian, Sur la conjecture combinatoire de Kashiwara-Vergne,
J. Lie Theory \textbf{12}, 597--616, (2002).

\bibitem{vergne} M. Vergne, Le centre de l'alg\`ebre enveloppante et la
formule de Campbell-Hausdoreff, C.R. Acad. Sci. Paris, S\'er. I, Math. \textbf{329},
767--772, (1999).


\end{thebibliography}
\end{document}